\documentclass{amsart}
\usepackage[utf8x]{inputenc}
\usepackage{amsmath,amsthm,amsfonts,amssymb,latexsym, enumerate}
\usepackage[matrix, arrow]{xy}
\xyoption{arrow}

\begin{document}
\thanks{Part of the  work on this paper was done  while the authors 
were visiting  the  University of California, Santa Cruz and they 
would like to thank  the Mathematics   Department  for its hospitality.
The second author acknowledges support from EPSRC grant
EP/M02525X/1.}

\theoremstyle{plain}
\newtheorem{thm}{Theorem}[section]
\newtheorem{lem}[thm]{Lemma}
\newtheorem{pro}[thm]{Proposition}
\newtheorem{cor}[thm]{Corollary}
\newtheorem{statement}[thm]{}

\theoremstyle{definition}
\newtheorem{que}[thm]{Question}
\newtheorem{rem}[thm]{Remark}
\newtheorem{defi}[thm]{Definition}
\newtheorem{Question}[thm]{Question}
\newtheorem{Conjecture}[thm]{Conjecture}
\newtheorem{exa}[thm]{Example}
\newtheorem{Notation}[thm]{Notation}

\def\CF{{\mathcal{F}}}  
\def\CG{{\mathcal{G}}}
\def\CH{{\mathcal{H}}}
\def\CO{{\mathcal{O}}}
\def\CP{{\mathcal{P}}}
\def\CR{{\mathcal{R}}}

\def\OD{{\mathcal{O}D}}
\def\OG{{\mathcal{O}G}}
\def\OGb{{\mathcal{O}Gb}}
\def\OHc{{\mathcal{O}Hc}}
\def\OH{{\mathcal{O}H}}
\def\ON{{\mathcal{O}N}}
\def\OP{{\mathcal{O}P}}
\def\OQ{{\mathcal{O}Q}}
\def\OR{{\mathcal{O}R}}

\def\Br{\mathrm{Br}} 
\def\dim{\mathrm{dim}}
\def\codim{\mathrm{codim}}
\def\coker{\mathrm{coker}}
\def\End{\mathrm{End}}
\def\Hom{\mathrm{Hom}}
\def\IBr{\mathrm{IBr}} 
\def\Im{\mathrm{Im}} 
\def\Ind{\mathrm{Ind}} 
\def\Irr{\mathrm{Irr}} 
\def\Id{\mathrm{Id}} 
\def\id{\mathrm{id}} 
\def\lcm{\mathrm{lcm}} 
\def\Aut{\mathrm{Aut}} 
\def\Ker{\mathrm{Ker}} 
\def\mod{\mathrm{mod}} 
\def\modh{{\operatorname{mod-\!}}}
\def\Mod{\mathrm{Mod}} 
\def\rank{\mathrm{rank}}
\def\Res{\mathrm{Res}} 
\def\op{\mathrm{op}}
\def\rk{\mathrm{rk}} 
\def\SL{\mathrm{SL}}
\def\Syl{\mathrm{Syl}} 
\def\Tr{\mathrm{Tr}} 
\def\tr{\mathrm{tr}} 
\def\Gal{\mathrm{Gal}} 
\def\ten{\otimes}
\def\tenA{\otimes_{A}}
\def\tenB{\otimes_{B}}
\def\tenO{\otimes_{\mathcal{O}}}
\def\tenOP{\otimes_{\mathcal{O}P}}
\def\tenOQ{\otimes_{\mathcal{O}Q}}
\def\tenOR{\otimes_{\mathcal{O}R}}
\def\tenK{\otimes_K}
\def\tenk{\otimes_k}
\def\tenkP{\otimes_{kP}}
\def\tenkR{\otimes_{kR}}

\def\C{{\mathbb C}}  
\def\F{{\mathbb F}}               
\def\Q{{\mathbb Q}}               
\def\Z{{\mathbb Z}}               
\def\Fp{{\mathbb F_p}}          
\def\Fq{{\mathbb F_q}}

\def\foc{{\mathfrak {foc}}}
%opening
\title[Descent of equivalences] {Descent of equivalences and character 
bijections}

\author{Radha Kessar, Markus Linckelmann}
\address{Department of Mathematics, City, University  of London 
EC1V 0HB, United Kingdom}
\email{radha.kessar.1@city.ac.uk}

\begin{abstract}
Categorical equivalences between block algebras of finite groups - such 
as Morita and derived equivalences - are well-known to induce character 
bijections which commute with the Galois groups of field extensions.
This is the motivation for attempting to realise known Morita and 
derived equivalences over non splitting fields. This article presents 
various results on the theme of descent to appropriate subfields and 
subrings. We start with the observation that perfect isometries induced 
by a virtual Morita equivalence induce isomorphisms of centers in 
non-split situations, and explain connections with Navarro's 
generalisation of the Alperin-McKay conjecture.  We  show that 
Rouquier's splendid Rickard complex for blocks with cyclic defect groups 
descends to the non-split case.  We also prove a descent theorem for  
Morita equivalences  with endopermutation source.
\end{abstract}

\maketitle

\bigskip

\begin{center}
\today
\end{center}

%%%%%%%%%%%%%%%%%%%%%%%%%%%%%%%%%%%%%%%%%%%%%%%%%%%%%%%%%%%%%%%%%%%%%%

\section{Introduction}     

Throughout the paper, $p$ is a prime number. Let $(K, \CO, k)$ be a
$p$-modular system; that is, $\CO$ is a complete discrete valuation ring 
with residue field $k=$ $\CO/J(\CO)$ of characteristic $p$ and field of 
fractions $K$ of characteristic zero. We are interested in capturing 
equivariance properties of various standard equivalences (such as  
Morita, Rickard or $p$-permutation  equivalences) in the block theory of 
finite groups. A common  context   for these is  the notion of virtual 
Morita equivalence which we recall.

Let $A$, $B$, $C$ be $\CO$-algebras, finitely generated free as 
$\CO$-modules. Denote by ${\modh}A$ the category of finitely generated  
left $A$-modules and by $\CR(A)$ the Grothendieck group of ${\modh}A$    
with respect to split exact sequences. Denoting by $[M]$ the element of 
$\CR(A)$  corresponding  to  the finitely generated $A$-module $M$,   
$\CR(A)$ is a free abelian group  with basis the set of all elements
$[M]$,  where $M$ runs through a set of representatives of the 
isomorphism classes of finitely generated indecomposable $A$-modules.    

Denote by $\CR(A,B)$ the group  $\CR(A\otimes_{\CO} B^{\op})$ and by 
$\CP(A, B)$ the subgroup of $\CR(A,B)$ generated by elements $[M]$,  
where $M$ is an $(A,B)$-bimodule which is finitely generated projective 
as left $A$-module and as right $B$-module. We let 
$-\cdot_B-:  \CR(A, B) \times \CR (B, C)  \to$ $\CR(A,C)$,
$M \times N \mapsto  M\cdot_BN$  be the  group homomorphism induced by  
tensoring over $B$, that is such that $[X]\cdot_B [Y]= [X\otimes_B Y]$  
for  all  finitely generated $(A, B)$-bimodules $X$ and 
$(B,C)$-bimodules $Y$. If $X$ is an $(A,B)$-bimodule, then its 
$\CO$-dual $X^\vee=\Hom_\CO(X,\CO)$ is a $(B,A)$-bimodule. The algebra 
$A$ is called {\it symmetric} if $A\cong$ $A^\vee$ as $(A,A)$-bimodules. 
If $A$ and $B$ are symmetric, and if $X$  is an $(A,B)$-bimodule which 
is finitely generated projective as left $A$-module and as right 
$B$-module, then the $(B, A)$-module $X^{\vee}$ is again finitely 
projective as left $B$-module and as right $A$-module (this holds
more generally if $A$ and $B$ are relatively $\CO$-injective). We denote 
in that case by $M \to M^{\vee}$ the unique homomorphism $\CP(A, B) \to$
$\CP(B,A)$ such that $[X]^{\vee}=$ $[X^{\vee}]$ for any $(A,B)$-bimodule
$X$ which is finitely generated projective as left $A$-module and as 
right $B$-module. Note that in the above we may replace $\CO$ by any 
complete local ring, and in particular by a field, and we will do so 
without  further comment.

\begin{defi}  
Let $A$ and $B$ be $\CO$-algebras, finitely generated free as 
$\CO$-modules, and let $M\in$ $\CP(A,B)$ and $N \in$ $\CP(B,A)$. We say 
that $M$ and $N$ induce a {\it virtual Morita equivalence between $A$ 
and $B$} if $M\cdot_B\ N  =$ $[A]$ in $\CR(A,A)$ and $N\cdot_A\ M=$
$[B]$ in $\CR(B,B)$.
\end{defi}

\begin{rem}\label{rem:equiv2}
Let $A$ and $B$ be symmetric $\CO$-algebras. We will use without further 
comment the following well-known implications between the various levels 
of equivalences we consider in this paper. If $M$ is an $(A,B)$-bimodule 
which is finitely generated projective as a left and right module and 
which induces a Morita equivalence between $A$ and $B$, then $[M]$ and 
$[M^\vee]$ induce a virtual Morita equivalence. More generally, if $X$ 
is a Rickard complex of $(A,B)$-bimodules, then $[X]=$
$\sum_{i\in\Z} (-1)^i[X_i]$ and $[X^\vee]$ induce a virtual Morita
equivalence between $A$ and $B$. Following \cite{BoXu}, if $A$, $B$ are 
blocks of finite group algebras, then a virtual Morita equivalence 
between $A$ and $B$ given by $M$ and $M^\vee$ is called a {\it 
$p$-permutation equivalence} if $M$ can be written in the form $M=$ 
$[M_0]-[M_1]$, where $M_0$, $M_1$ are $p$-permutation $(A,B)$-bimodules 
which are finitely generated projective as left and right modules. In 
particular, if $X$ is a splendid Rickard complex of $(A,B)$-bimodules, 
then $[X]$ and $[X^\vee]$ induce a $p$-permutation equivalence. 
\end{rem}

Let $K'$ be an extension field of $K$. For an $\CO$-algebra $A$, we 
denote by $K'A$ the $K'$-algebra $K'\otimes_{\CO} A$, and for any
$A$-module $V$ we denote by $K'V$ the $K'A$-module $K'\tenO V$. 
The  functor $K'\otimes_{\CO} - : \modh A \to$ $\modh K'A$ induces a 
group homomorphism $[V] \mapsto [K'V]$ from $\CR(A) $  to $\CR(K'A)$, 
for all finitely generated $A$-modules $V$. We use analogous notation 
for bimodules. Let $\Aut(K'/K)$  denote the group of automorphisms of 
$K'$ which induce the identity on $K$. For $\sigma\in$ $\Aut(K'/K)$ and 
a $K'A$-module $U$ we denote by $\,^{\sigma}U$ the $\sigma $-twist of  
$U$, that is  $\, ^{\sigma}U$ is the $K'A$-module which is equal to $U$ 
as a $KA$-module and on which  $\lambda \otimes a $  acts as 
$\sigma^{-1}(\lambda) \otimes  a$ for all $\lambda \in K'$ and all 
$a \in KA $. We use the analogous notation for the induced map on   
$\CR(K'A)$.

If $K'A$ is a semisimple algebra, we denote by $\Irr(K'A)$ the subset  
of $\CR(K'A)$ consisting of the elements $[S]$, where $S$ runs through a 
set of representatives of isomorphism classes of simple $K'A$-modules. 
Then  $\Irr(K'A)$ is a $\Z$-basis of $\CR(K'A)$. For $\chi=[S] \in $
$\Irr (K'A)$ we denote by $e_{\chi}$ the unique primitive idempotent of 
$Z(K'A)$ such that $ e_{\chi}  S \ne 0 $.  

The following general result on symmetric $\CO$-algebras is the starting
point of the phenomenon we wish to exhibit.  

\begin{thm} \label{thm:introvirtual}   
Let $(K, \CO, k)$ be a $p$-modular system.
Let $A$ and $B$ be  symmetric $\CO$-algebras and let  $K'$ be an 
extension field of $K$ such that $K'A$ and $K'B$ are split semisimple.
Suppose that $M \in \CP(B,A)$ is such that $M$ and $M^{\vee} $ induce 
a virtual Morita equivalence between $A$ and $B$.     
Then there exists a bijection $I : \Irr(K'A) \to \Irr(K'B)$ and signs 
$\epsilon_{\chi} \in  \{\pm 1\}$ for any $\chi\in$ $\Irr(K'A)$ such that 
$$\epsilon_{\chi} I(\chi) =  K'M\cdot_{K'A}\ \chi$$
for all $ \chi \in \Irr(K'A) $. Moreover, the following holds.

\begin{enumerate} [(a)]  
\item 
The algebra isomorphism $Z(K'A)\cong$ $Z(K'B)$ sending $e_{\chi}$ to 
$e_{I(\chi)} $ for all $\chi\in$ $\Irr(K'A)$  induces an $\CO$-algebra 
isomorphism  $Z(A) \cong Z(B) $.

\item     
The bijection $I$ commutes with $\Aut(K'/K)$; that is, we have 
$I(\,^{\sigma}\chi)= \,^\sigma  I (\chi)$ for all $\sigma \in$
$\Aut(K'/K)$ and all $\chi \in$ $\Irr(K'A)  $.  
\end{enumerate}
\end{thm}

\begin{rem} \label{rem:equiv1}
By a result of Brou\'e \cite[1.2]{Broue90}, a virtual Morita equivalence 
between two blocks of finite group algebras over $\CO$ given by a 
virtual bimodule and its dual induces a perfect isometry. In particular,
if $A$ and $B$ in Theorem \ref{thm:introvirtual} are blocks of finite 
group algebras, then the bijection $I$ and the signs in the theorem are 
together a perfect isometry. Not every  perfect isometry is induced by a 
virtual Morita  equivalence, but the advantage of virtual Morita 
equivalences is that they are  defined for arbitrary algebras. 
Brou\'e's abelian defect group conjecture, in the version predicting
a Rickard equivalence between a block with an abelian defect group
and the Brauer correspondent of that block, in conjunction with
Remark \ref{rem:equiv2}, implies therefore that the induced perfect 
isometry between a block and its Brauer correspondent is in fact 
induced by a virtual Morita equivalence. 
\end{rem}

We are interested in subgroups of $\Gal(\Q_n/\Q)$  which lift 
automorphisms in  characteristic $p$. For a positive integer $n$, we 
denote by $n_p$ (respectively $n_{p'}$) the  $p$-part (respectively 
$p'$-part) of $n $.

\begin{defi} 
Let $(K,\CO, k)$ be a $p$-modular system such that $k$ is perfect.
Let $\bar K$ be an algebraic closure of $K$. Let $n$ be a positive 
integer, denote by $\Q_n$ the  $n$-th cyclotomic subfield of $\bar K$, 
and let $k'$ be a splitting field of the polynomial $x^{n_{p'}}-1$ over 
$k$. We denote by $\CH_n$ the subgroup of $\Gal(\Q_n/\Q)$ consisting of 
those automorphisms $\alpha $ for which there exists a non-negative 
integer $u$ such that $\alpha(\delta) =$ $\delta^{p^u}$ for all 
$n_{p'}$-roots of unity $\delta$ in $\Q_n$. We denote by $\CH_{n,k}$ the 
subgroup of $\CH_n$ consisting of those automorphisms $\alpha $ for 
which there exists a non-negative integer $u$ and an element  $\tau \in 
\Gal(k'/k)  $ such that $\alpha (\delta) =$ $\delta^{p^u}$   for all 
$n_{p'}$-roots of unity $\delta$ in $\Q_n$  and $\tau( \eta) =$ 
$\eta^{p^u}$  for all $n_{p'} $-roots of  unity  $\eta $ in $k'$.

%there exists a non-negative integer $u$ such that $\sigma(\delta) =$
%$\delta^{p^u}$ for all $p'$-roots of unity $\delta$ in $\Q_n$.  Let $\CH_{n,k} $  
%be the subgroup  of  ${\mathcal H}_n $  consisting of those $\sigma $ in ${\mathcal H}_n $  
%such that  $\sigma (\delta) =  \delta  $ for all  $p'$-roots  of unity  in ${\mathbb Q}_n \cap  \CO $.
%If $|k|=p^u$, let $\CH_{n,k}$ be the subgroup of ${\mathcal H}_n $  
%generated by automorphisms $\sigma $ such that  $\sigma(\delta)=$
%$\delta^{p^u}$ for every $p'$-root of unity $\delta$ in ${\mathbb Q}_n$.
%If $k$ is infinite, let $\CH_{n,k}$ be the subgroup of ${\mathcal H}_n$
%consisting of automorphisms which act as identity on all $p'$-roots of 
%unity in ${\mathbb Q}_n$. 
\end{defi}

Note that  $\CH_n $ is the image  under restriction in ${\mathbb Q}_n $ of the  Weil subgroup of the  absolute Galois group  of  the $p$-adic numbers (see Lemma~\ref{lem:stablering} and  Lemma~\ref{lem:cyclotomic}). Note also that   $\CH_{n,k} $ is independent of the choice of a splitting field  
$k'$ of $x^{n_{p'}}-1$ over $k$. With the notation of the above 
definition, for a finite group $G$, the set  
$\Irr(\bar K G)$ may be identified with the set of characters 
$\chi: G \to \bar K$ of simple $\bar KG$-modules. The group  
$\Aut(\bar K)$ acts on $\Irr(\bar K G)$ via  
$\,^{\sigma} \chi(g) := \chi(\sigma(g)) $,  $\chi \in $
$\Irr(\bar K G) $, $\sigma \in \Aut(\bar K)$. We say that {\it a 
positive integer $n$ is large enough for $G$} if the action of 
$\Aut(K) $ on $\Irr(\bar K G)$ factors through to an action of 
$\Gal(\Q_n/\Q)$ via the surjective homomorphism from $\Aut(\bar K)$ to 
$\Gal(\Q_n/\Q)$ induced by restriction to $\Q_n$. In particular, if $n$ 
is a multiple of $|G|$, then $n$ is large enough for $G$.

By a {\it block} of  $\CO G $  for $G$ a finite group we mean a  
primitive idempotent of the  center of the group algebra $\CO G $. If 
$b$ is a block of $\CO G$ we  denote by $\Irr(\bar KGb)$ the subset of 
$\Irr(\bar KG)$ consisting of the characters of simple 
$\bar K Gb $-modules. There are many open questions and conjectures    
around bijections between sets of irreducible characters of blocks which 
commute with the action of the groups $\CH_n$ and $\CH_{n,k} $, most 
notably Navarro's refinement of  the Alperin-McKay conjecture  
\cite[Conjecture B]{Nav04}.
 
Theorem \ref{thm:introvirtual} yields the following equivariance result 
for character bijections. The slogan is: categorical equivalences 
between blocks over absolutely unramified complete discrete valuation 
rings give rise to character bijections which commute with  the action  
of $\CH_{n,k} $.  Recall that  $\CO$ is said to be {\it absolutely 
unramified} if $J(\CO) = p\CO$.

\begin{thm}  \label{thm:intro-basic}    
Let $(K, \CO, k)$ be a $p$-modular system such that $k$ is perfect
and such that $\CO$ is absolutely  unramified. Let $\bar K$ be an 
algebraic closure of $K$. Let $G$ and $H$ be finite 
groups, let $b$ be a block of $\CO G $, $c$ a block of $\CO H$ and let  
$n$ be large enough  for $G$ and for $H$. A virtual Morita equivalence 
between  $\CO Gb$ and $\CO Hc$  given by a virtual bimodule $M$ and its 
dual $M^\vee$ induces a bijection  $ I : \Irr(\bar KGb) \to$
$\Irr(\bar KHc) $ satisfying  $\,^{\sigma} I(\chi) =$
$I (\,^{\sigma} \chi) $ for all $\chi \in$ $\Irr(\bar KGb)$ and all 
$\sigma \in$ $\CH_{n,k}$.  
\end{thm} 

As mentioned before, the bijection $I$ in the above theorem is part of 
a perfect isometry.  Further, Morita, Rickard and $p$-permutation 
equivalences all yield virtual Morita equivalences. Thus the conclusion 
of the Theorem holds on replacing the hypothesis of virtual Morita 
equivalence by any of these equivalences - in the case of a Morita 
equivalence the induced  bijection between the sets of irreducible 
Brauer characters also  commutes  with the action of $\CH_{n,k} $ as 
well as with the decomposition map (see Theorem ~\ref {thm:rational}).

Recall that a character  $\chi \in$ $\Irr(\bar KG)$ is said to be 
{\it $p$-rational} if  there  there exists a  root  of unity $\delta$ 
in $\bar K$  of order prime to $p$ such that $ \chi(g) \in$ $\Q[\delta]$ 
for all $g\in G$.   Theorem~\ref{thm:intro-basic} has the following 
consequence.

\begin{cor} \label{cor:intro-rational}   
Suppose that $\CO$ and $\bar K$ are as in Theorem \ref{thm:intro-basic}.    
Any virtual Morita equivalence between block algebras $\CO Gb$ and 
$\CO Hc$  given by a virtual bimodule and its dual induces a  bijection 
between $\Irr(\bar KGb) $ and $\Irr(\bar KHc)$ which preserves 
$p$-rationality.
\end{cor}

Recall that for a perfect subfield  $k'$  of $k$, there is a unique    
absolutely unramified complete discrete valuation ring $W(k')$ 
contained in $\CO$ such that the image of $W(k')$ under the canonical  
surjection $\CO \to k$ is $k'$  (see \cite[Chapter 2, Theorems 3, 4  
and Prop.~10]{SeLF}). The ring $W(k')$ is called the ring of Witt  
vectors  in $\CO$ of $k'$.  

\begin{defi}  
Let $(K, \CO, k)$ be a $p$-modular system. 
Let  $G$ be a finite group and  $b$ a  block of $\CO G$. The {\it 
minimal complete discrete valuation ring of $b$ in $\CO$} denoted 
$\CO_b$ is the ring of Witt vectors in $\CO$ of the finite subfield  
of $k$ generated by the coefficients of the group elements in the 
image of $b$ under the  canonical surjection $\CO G \to kG$. If 
$\CO_b= \CO$, then we say that $\CO$ {\it is a minimal complete 
discrete valuation  ring  of}  $b$.
\end{defi}

By idempotent lifting arguments we have  $b \in \CO_b G $, and if $R$ is 
any complete discrete valuation ring which is properly contained in 
$\CO_b$ and with $J (R) \subseteq J(\CO_b)$, then $b \notin RG$.

The following is a corollary of the special case of 
Theorem~\ref{thm:intro-basic}  in which $\CO$ is a minimal complete 
discrete  valuation ring of the  blocks involved.    

\begin{cor}   \label{cor:intro-virtual}  
Suppose that $\CO$ and $\bar K$ are as in Theorem \ref{thm:intro-basic}.    
Let $G$ and $H$ be finite groups and let $n$ be large enough for $G$ 
and for $H$. Let $b$ be a block of $\CO G$ and $c$ a block of $\CO H$.
Suppose that $\CO$ is a minimal complete discrete valuation ring for both 
$b$ and $c$.  A virtual Morita equivalence between  $\CO Gb $ and 
$\CO  Hc $ given by a virtual bimodule and its dual induces
a bijection $I : \Irr(\bar KGb) \to \Irr(\bar KHc) $ such that for   
any $\chi \in \Irr (\bar KGb) $, and  any $\sigma\in\CH_n$, we have 
 $\,^\sigma \chi = \chi$ if and only if $\,^{\sigma}I(\chi) =I(\chi)$. 
 \end{cor} 

For a $p$-subgroup $P$ of $G$  the Brauer homomorphism 
$\Br_P:  (\CO G)^P  \to kC_G(P)$ is the map  which sends an element 
$\sum_{g\in G} \alpha_g g$ of $(\CO G )^P$ to 
$\sum_{g\in C_G(P)} \bar \alpha_g g $, where $\bar \alpha $ denotes 
reduction modulo the maximal ideal $J(\CO )$ of $\CO $. Recall that 
$\Br_P$ is a  surjective $\CO$-algebra  homomorphism and that 
$\Br_P  (Z(\CO G)) \subseteq Z(kC_G(P))$. In particular, if $b$ is a 
central idempotent of $\CO G$, then either $\Br_P(b) =0 $ or $\Br_P(b)$  
is a central idempotent of $kC_G(P)$. If $b$ is a block of $\CO G$,   
then a  defect group of $b$  is defined to be a maximal $p$-subgroup  
$P$ of $G$ such that  $\Br_P (b) \ne 0 $.  By Brauer's first main 
theorem, if $b$ is a block of $\CO G$ with  defect group $P$, then 
there is a unique block $c$ of $\CO N_G(P)$ with defect group $P$ such 
that $\Br_P(b) =$ $\Br_P(c)$  and the map $b \mapsto c$ is  a  
bijection between  the set of blocks  of $\CO G $  with defect group 
$P$ and the set of blocks of $\CO N_G(P)$  with defect group $P$,  
and this bijection is known as the  Brauer correspondence.  

In \cite[Conjecture B]{Nav04}, Navarro conjectured that if $|G|=n$, 
$b$ and $c$ are blocks in correspondence as above and  $K$ 
contains $\Q_n$,  then for each $\sigma \in \CH_n $ the number of 
height zero characters in $\Irr(\bar KGb)$ fixed by $\sigma $ equals 
the   number of height zero characters in $\Irr(\bar K H c)$ fixed by 
$\sigma $. Since $\CO_b =\CO_c $, and since the bijection  $I$  of  
Corollary~\ref{cor:intro-virtual} is part of a perfect isometry and 
hence preserves heights,  it follows that a virtual  Morita equivalence 
between   $\CO_b Gb $ and $\CO_b H  c $ given by a virtual bimodule and
its dual implies  Navarro's  conjecture.  
 
In view of the above discussion, it would  be  desirable to  explore 
the following question: Given a categorical equivalence, say a Morita 
equivalence  or Rickard equivalence between $\CO' Gb $ and $\CO' Hc$    
for some complete discrete valuation ring $\CO'$, for $G$ and $H$ finite 
groups, $b$ and $c$ blocks  of $\CO' G $ and $\CO' H$ respectively, and  
a complete discrete valuation ring  $\CO$  contained in $\CO' $ such 
that $b $ (respectively $c$)  belongs to $\CO Gb$  (respectively 
$\CO Hc $), is the equivalence between $\CO' Gb$ and $\CO' Hc$ an 
extension of an equivalence  between  $\CO b $ and $\CO Hc $? We give a 
positive answer to this question in the case of blocks with cyclic 
defect groups.

Let $G$  be a finite group and $b$ a block of $\CO G$ with a 
nontrivial cyclic defect group $P$. If $k$ is a splitting field for all 
subgroups of $G$, then  in  \cite{Rouqcyclic} Rouquier constructed a 
$2$-sided splendid tilting complex of $(\CO Gb, \CO N_G(P)e)$-bimodules, 
where $e$ is the Brauer correspondent of $b$. (The hypotheses in 
\cite{Rouqcyclic} also require the field of fractions  $K$ to be large 
enough, but it is easy to see that Rouquier's construction works with 
$\CO$ absolutely unramified). We show  that Rouquier's construction  
descends to any $p$-modular system which contains the block 
coefficients.

\begin{thm} \label{cyclic-splendid-intro}   
Let $(K',\CO',k')$ be a $p$-modular system such that $\CO\subseteq$ 
$\CO'$ and such that $J(\CO)\subseteq$ $J(\CO')$. Let $G$ be a finite 
group and $b$ a block of $\CO' G$ having a nontrivial cyclic defect 
group $P$. Suppose that $b\in$ $\OG$  and that $k'$ is a splitting field 
for all subgroups of $G$. Let $e$ be the block of $\CO' N_G(P)$ with $P$ 
as a defect group corresponding to $b$ via the Brauer correspondence. 
Then $e\in$ $\CO N_G(P)$ and the blocks $\OGb$ and $\CO N_G(P)e$ are 
splendidly Rickard equivalent. More precisely, there is a splendid 
Rickard complex $X$ of $(\OGb, \CO N_G(P)e)$-bimodules such that 
$\CO' \tenO X$ is isomorphic to Rouquier's complex $X'$.
\end{thm}

Since a Rickard equivalence induces a virtual Morita equivalence, by the 
above discussion around Navarro's conjecture, we recover the following
result of Navarro from Theorem~\ref{cyclic-splendid-intro}.
 
\begin{cor}[{\cite[Theorem 3.4]{Nav04}}] 
Conjecture B of \cite{Nav04} holds for blocks with cyclic defect groups.
 \end{cor} 

General descent arguments from Theorem \ref{thm:further-descent}
in conjunction with Theorem \ref{cyclic-splendid-intro} yield a splendid
equivalence for cyclic blocks for arbitrary $p$-modular systems.

\begin{thm} \label{cyclic-splendid-general}   
Let $(K,\CO,k)$ be a $p$-modular system.
Let $G$ a finite group and $b$ a block of  $\CO G$ having a nontrivial 
cyclic defect group $P$. Let $e$ be the block of $\CO N_G(P)$ with $P$ 
as a defect group corresponding to $b$ via the Brauer correspondence. 
Then $\CO Gb $ and  $\CO N_G(P)e $ are splendidly Rickard equivalent.
In particular, $\OGb$ and $\CO N_G(P)e$ are $p$-permutation equivalent.
\end{thm}

The above results may be viewed as evidence for a refined version of the
Abelian defect group conjecture, namely that for any $p$-modular system
$(K, \CO, k)$ and any block $b$  of $\CO G$ with abelian defect group  
$P$ and Brauer correspondent $c$, there is a splendid Rickard 
equivalence between $\CO G b $ and   $\CO N_G(P)  c $.

If one is only interested in keeping track of $p$-rational characters, 
then by Corollary \ref{cor:intro-rational} it suffices to descend to   
any  absolutely unramified complete discrete valuation ring. Since    
$p$-permutation modules all have forms over absolutely unramified 
complete discrete valuation rings, any $p$-permutation equivalence 
between  block algebras of finite groups can be easily seen to be an 
extension of a $p$-permutation equivalence  between the corresponding   
blocks over the subring of Witt vectors.  We show that such descent is 
also possible for Morita equivalences induced by bimodules with
endopermutation sources.

\begin{thm}  \label{thm:basicunramified}
Let $(K, \CO, k)$ be a $p$-modular system. Let $G$ and $H$ be finite 
groups, $b$ a block of $\OG$ and $c$ a block of $\OH$. Denote by
$\bar b$ the image of $b$ in $kG$ and by $\bar c$ the image of $c$ in 
$kH$. Assume that $k$ is a splitting field for all subgroups of
$G\times H$.

\begin{enumerate} [(a)]  
\item 
For any Morita equivalence (resp. stable equivalence of Morita type)
between $kG\bar b$ and $kH\bar c$ given by an indecomposable  
bimodule $\bar M$ with endopermutation source $\bar V$ there 
is a Morita equivalence (resp. stable equivalence of Morita type) 
between $\OGb$ and $\OHc$ given by a bimodule $M$ with endopermutation 
source $V$ such that $k\tenO M\cong$ $\bar M$ and $k\tenO V\cong$ 
$\bar V$.

\item 
For any Morita equivalence (resp. stable equivalence of Morita type)
between $\OGb$ and $\OHc $ given by an indecomposable bimodule with 
endopermutation source $V$  there is a Morita equivalence
(resp. stable equivalence of Morita type) between $W(k)Gb$ and 
$W(k)Hc$ given by an indecomposable bimodule with endopermutation 
source $U$ such that  $k\ten_{W(k)} U\cong$ $k\tenO V$.
\end{enumerate} 
\end{thm}  

\begin{rem} 
The proof of  the  above theorem requires a lifting property  
of fusion stable endopermutation modules from Lemma \ref{lem:stab} 
below, which in turn relies on the classification of endopermutation
modules. The hypothesis on $k$ being large enough is there to
ensure that the fusion systems of the involved blocks are saturated.
The well-known Morita equivalences in block theory such as in the 
context of nilpotent blocks \cite{Punil}, blocks with a normal defect 
group \cite{Kuenormal} and blocks of finite $p$-solvable groups 
\cite{Kue81}, \cite{Puigsolvable}  are all  given  by endopermutation source bimodules  
hence are defined over the Witt 
vectors   and preserve $p$-rational characters and $p$-rational lifts 
of Brauer characters (cf.  Corollary \ref{cor:intro-rational}, Theorem 
\ref{thm:rational}).
\end{rem}  

The paper is organised as follows. Section \ref{virtualMorita-Section}
contains the proof of Theorem~\ref{thm:introvirtual} and Section 
\ref{Galois-Section} contains the proofs of 
Theorem~\ref{thm:intro-basic} and its corollaries. Sections 
\ref{descentequiv-Section}, \ref{descentrelproj-Section} and 
\ref{descentGalois-Section} contain general results on  descent. 
Theorems~\ref{cyclic-splendid-intro}  and \ref{cyclic-splendid-general} 
are proved in Section \ref{cyclic-Section}, and Section 
\ref{basic-Section} contains the proof  of 
Theorem~\ref{thm:basicunramified}.

\begin{Notation}
We will use the above notation  of Galois twists for arbitrary 
extensions of commutative rings $\CO\subseteq$ $\CO'$. That is, given an 
$\CO$-algebra $A$, a module $U$ over the $\CO'$-algebra $A'=$ 
$\CO'\tenO A$ and a ring automorphism $\sigma$ of $\CO'$  which 
restricts to the identity map on $\CO$, we denote  by ${^\sigma{U}}$ 
the $A'$-module which is equal to $U$ as a module over the subalgebra  
$1\ten A$ of $A'$, such that $\lambda\ten a$ acts on $U$ as  
$\sigma^{-1}(\lambda)\ten a$ for all $a\in$ $A$ and $\lambda\in$ $\CO'$. 
Note that if $f : U\to$ $V$ is an $A'$-module homomorphism, then $f$ is 
also an $A'$-module homomorphism ${^\sigma{U}}\to$ ${^\sigma{V}}$. The 
Galois twist induces an $\CO$-linear (but not in general $\CO'$-linear) 
self equivalence on $\modh A'$. 
\end{Notation}

 %%%%%%%%%%%%%%%%%%%%%%%%%%%%%%%%%%%%%%%%%%%%%%%%%
\section{On virtual Morita equivalences} \label{virtualMorita-Section}
  
This section contains the proof of Theorem \ref{thm:introvirtual}.  We 
start with some background observations. Let $(K,\CO,k)$ be a 
$p$-modular system. 
It is well-known that a virtual Morita equivalence between two split
semisimple algebras given by a virtual bimodule and its dual is
equivalent to fixing a bijection between the isomorphism classes of
simple modules of the two algebras together with signs. We sketch the
argument for the convenience of the reader.

\begin{lem} \label{signsLemma}
Let $A$ and $B$ be split semisimple finite-dimensional $K$-algebras.
Let $M$ be a virtual $(A,B)$-bimodule in $\CR(A,B)$. Then $M$ and
$M^\vee$ induce a virtual Morita equivalence between $A$ and $B$ if
and only if there is a bijection $I : \Irr(A)\to$ $\Irr(B)$ and signs 
$\epsilon_S\in$ $\{\pm 1\}$ for all $S\in$ $\Irr(A)$ such that
$$M = \sum_{S\in\Irr(A)} \epsilon_S S \tenK  I(S)^\vee$$
in $\CR(A,B)$. 
\end{lem}

\begin{proof}
Write $M=$ $\sum_{S,T} a(S,T) S\tenK T^\vee$, with integers $a(S,T)$,
where $S$ and $T$ run over $\Irr(A)$ and $\Irr(B)$, respectively.
Since $B$ is split semisimple, we have $T^\vee\tenB T\cong$ $K$ and
$T^\vee\tenB T'=$ $\{0\}$, where $T$, $T'\in$ $\Irr(B)$, $T\not\cong$ 
$T'$. Thus $M \cdot_B\ M^\vee=$ 
$\sum_{S,S',T} a(S,T)a(S',T) S'\tenK S^\vee$,
with $S$, $S'$ running over $\Irr(A)$ and $T$ running over $\Irr(B)$.
We have the analogous formula for $M^\vee \cdot_A\ M$. Since
$A$ is split semisimple, we have $[A]=$ $\sum_{S}\ S\tenK S^\vee$.
Thus $M$, $M^\vee$ induce a virtual Morita equivalence if and only if
$\sum_{T} a(S,T)^2=1$ for all $S\in$ $\Irr(A)$, and
$\sum_{T} a(S,T)a(S',T) =0$ for any two distinct $S$, $S'$ in $\Irr(A)$. 
Since the $a(S,T)$ are integers, the first equation implies that for any
$S$ there is a unique $T=$ $I(S)$ such that $a(S,T)\in$ $\{\pm 1\}$ and
$a(S,T')=0$ for $T'\neq$ $T$. The second equation implies that $I$ is a
bijection. The result follows with $\epsilon_S=$ $a(S,I(S))$.  
\end{proof}

We will use the transfer maps in Hochschild cohomology from
\cite{Lintransfer}, specialised in degree $0$; we sketch the 
construction. 
Let $A$ and $B$ be symmetric $\CO$-algebras with fixed symmetrising
forms. Let $M$ be an  $(A,B)$-bimodule which is finitely generated 
projective as left  $A$-module and as right $B$-module. Then the
functors $M\tenB -$ and $M^\vee \tenA-$ are biadjoint; the choice of
the symmetrising forms determines adjunction isomorphisms. Let $y\in$ 
$Z(B)$. Multiplication by $y$ induces a $(B, B)$-bimodule
endomorphism of $B$. Tensoring by $M\tenB - \tenB M^\vee$ yields
an $A$-$A$-bimodule endomorphism of $M\tenB M^\vee$. 
Composing and precomposing this endomorphism by the
adjunction counit $M\tenB M^\vee\to$ $A$ and the adjunction unit 
$A\to$ $M\tenB M^\vee$ yields an $(A, A)$-bimodule endomorphism
of $A$, which in turn yields a unique element $z\in$ $Z(A)$ which 
induces this endomorphism by multiplication on $A$.  We define
the linear map $\tr_M : Z(B)\to$ $Z(A)$ by setting $\tr_M(y) = z$,
with $y$ and $z$ as above. The map $\tr_M$ is additive in $M$ (cf.
\cite[2.11.(i)]{Lintransfer}), depends only on the isomorphism class
of $M$ (cf. \cite[2.12.(iii)]{Lintransfer})  and is compatible with 
tensor products of bimodules (cf. \cite[2.11.(ii)]{Lintransfer}). 
In general, $\tr_M$ depends on the choice of the symmetrising forms 
(because the adjunction units and counits depend on this choice), but 
there is one case where it does not:

\begin{lem} \label{transferA}
Let $A$ be a symmetric $\CO$-algebra. Consider $A$ as an 
$(A, A)$-bimodule. Then $\tr_A =$ $\Id_{Z(A)}$.
\end{lem}

\begin{proof}
Let $s : A\to$ $\CO$ be a symmetrising form of $A$, and let $X$ be an 
$\CO$-basis of $A$. Denote by $X'$ the dual basis of $A$ with respect to
$s$; for $x\in$ $X$, denote by $x'$ the unique element in $X'$ satisfying
$s(xx')=1$ and $s(yx')=0$, for all $y\in$ $X\setminus \{x\}$. 
The well-known explicit description of the adjunction maps (see e. g. 
\cite[Appendix]{Lintransfer})
implies that the adjunction unit $A\to$ $A\tenA A^\vee$ sends $1_A$ to 
$1_A\ten s$, and the adjunction counit $A\tenA A^\vee\to$ $A$ sends
$1_A\ten s$ to $\sum_{x\in X}\ s(x')x$, which is equal to $1_A$ by
\cite[Appendix 6.3.3]{Lintransfer}. One can prove this also without those
explicit descriptions, by first observing that the above adjunction maps 
are isomorphisms, and  deduce from this that $\tr_A$ is a linear 
automorphism.  Since $\tr_A\circ\tr_A=$ $\tr_{A\tenA A}=$ $\tr_A$, this 
implies that $\tr_A=$ $\Id_{Z(A)}$. 
\end{proof}

In order to show that $\tr_M$ is well-defined with $M$ replaced by any
element in the Grothendieck group $\CP(A,B)$, we need the following
observation.

\begin{lem} \label{transferGroth}
Let $A$, $B$ be symmetric $\CO$-algebras with chosen symmetrising
forms. Let $M_0$, $M_1$, $N_0$, $N_1$ be $(A,B)$-bimodules which are 
finitely generated projective as left and as right modules. If 
$[M_0]-[M_1]=$ $[N_0]-[N_1]$ in $\CP(A,B)$, then $\tr_{M_0}-\tr_{M_1}=$
$\tr_{N_0}-\tr_{N_1}$. 
\end{lem}

\begin{proof}
The equality  $[M_0]-[M_1]=$ $[N_0]-[N_1]$ is equivalent to
$[M_0\oplus N_1]=$ $[N_0\oplus M_1]$. The Krull-Schmidt Theorem
implies that this is equivalent to $M_0\oplus N_1\cong$ $N_0\oplus M_1$.
The additivity of transfer maps implies that in that case we have
$\tr_{M_0} + \tr_{N_1} =$ $\tr_{N_0} + \tr_{M_1}$, whence the result.
\end{proof}

This Lemma implies that if $A$, $B$ are symmetric $\CO$-algebras with 
chosen symmetrising forms, then for any $M\in$ $\CP(A,B)$ we have a 
well-defined map $\tr_M : Z(A)\to$ $\Z(B)$ given by $\tr_M=$ 
$\tr_{M_0}-\tr_{M_1}$, where $M_0$, $M_1$ are $(A,B)$-bimodules which 
are finitely generated as left and right modules such that $M=$ 
$[M_0]-[M_1]$.

\begin{lem} \label{transfertensor}
Let $A$, $B$, $C$ be symmetric $\CO$-algebras. Let $M\in$ $\CP(A,B)$ and
$N\in$ $\CP(B,C)$. Then $\tr_{M\cdot_{B}\ N} =$ $\tr_M\circ \tr_N$. In 
particular, if $M$ and $M^\vee$ induce a virtual Morita equivalence 
between $A$ and $B$, then $\tr_M : Z(A)\to$ $Z(B)$ is a linear 
isomorphism with inverse $\tr_{M^\vee}$.
\end{lem}

\begin{proof}
The first equality follows from the corresponding equality 
\cite[2.11.(ii)]{Lintransfer} where $M$ and $N$ are actual bimodules, 
together with \ref{transferGroth}.
The second statement follows from the first and \ref{transferA}.
\end{proof}

\begin{rem} 
The three lemmas \ref{transferA}, \ref{transferGroth}, 
\ref{transfertensor} hold verbatim for the transfer maps on the 
Hochschild cohomology of $A$ in $B$ in any non-negative degree, and
with $\CO$ replaced by any complete local principal ideal domain.
\end{rem}

\begin{proof}[{Proof of Theorem \ref{thm:introvirtual}}]
We use the notation and hypotheses from Theorem \ref{thm:introvirtual}. 
Write $M=$ $[M_0]-[M_1]$, where $M_0$, $M_1$ are $(A,B)$-bimodules
which are finitely generated projective as left $A$-modules and as right 
$B$-modules.  By \ref{signsLemma} there exist a bijection 
$I : \Irr(K'A)\to$ $\Irr(K'B)$ and signs $\epsilon_\chi\in$ $\{\pm 1\}$ 
such that $\epsilon_\chi \chi=$ $K'M \cdot_{K'B}\  I(\chi)$ in 
$\CR(K'A)$ for all $\chi\in$ $\Irr(K'A)$. By \ref{transfertensor}, the 
linear map $\tr_M : Z(B)\to$ $Z(A)$ is an isomorphism, with inverse 
$\tr_{M^\vee}$. Let $v\in Z(B)$ such that $\tr_M(v)=1_A$ and let
$u\in$ $Z(A)$ such that $\tr_{M^\vee}(u)=1_B$. Define linear maps
$\alpha : Z(A)\to$ $Z(B)$ and $\beta : Z(B)\to$ $Z(A)$ by setting
$$\alpha(z) = \tr_{M^\vee}(uz)$$
$$\beta(y) = \tr_M(vy)$$
for all $z\in$ $Z(A)$ and $y\in$ $Z(B)$. By the choice of $u$ and $v$ we 
have $\alpha(1_A)=1_B$ and $\beta(1_B)=1_A$. We extend $\alpha$ and 
$\beta$ $K'$-linearly to maps, still called $\alpha$, $\beta$, between 
$Z(K'A)$ and $Z(K'B)$. Setting $K'M=$ $K'\tenO M$ as before, note that 
the transfer map $\tr_{K'M} : Z(K'B)\to$ $Z(K'A)$ is the $K'$-linear 
extension of $\tr_M$. Note further that $K'M = $ 
$\sum_{\chi\in\Irr(K'A)} \epsilon_\chi \chi\ten_{K'} I(\chi)^\vee $ and 
$\chi\ten_{K'}\ I(\chi)^\vee=$ $e_\chi K'M e_{I(\chi)}$. Thus 
$\tr_{K'M} =$ 
$\sum_{\chi\in\Irr(K'A)}\ \epsilon_\chi \tr_{{e_\chi} K'M e_{I(\chi)}}$.
In particular, $\tr_{K'M}$ sends $K'e_\chi$ to $K'e_{I(\chi)}$. 
 
Let $\chi\in$ $\Irr(K'A)$ and $\eta=$ $I(\chi)$. We have 
$\beta(e_\eta)=$ $\tr_{K'M}(ve_\eta)$. Since $Z(K'B)$ is a direct 
product of copies of $K'$, it follows that $v e_\eta=$ 
$\lambda_\chi e_\eta$ for some $\lambda_\chi\in$ $K'$. Thus 
$\tr_{K'M}(ve_\eta)=$ $\mu_\chi e_\chi$ for some $\mu_\chi\in$ $K'$. 
Therefore
$$1_A = \tr_{M}(v) = \sum_{\chi\in\Irr(K'A)} \ \tr_{K'M}(v e_{I(\chi)}) 
= \sum_{\chi\in\Irr(K'A)}\ \mu_\chi e_\chi\ .$$
Since also $1_A=$ $\sum_{\chi\in\Irr(K'A)}\ e_\chi$, the linear
independence of the $e_\chi$ implies that all $\mu_\chi$ are $1$,
hence that $\alpha(e_\chi)=$ $e_{I(\chi)}$. This shows that $\alpha$ and
$\beta$ are inverse algebra isomorphisms $Z(K'A)\cong$ $Z(K'B)$. By 
their constructions, $\alpha$ maps $Z(A)$ to $Z(B)$ and $\beta$ maps 
$Z(B)$ to $Z(A)$. This proves statement (a). This shows also that the
isomorphism $Z(K'A)\cong$ $Z(K'B)$ sending $e_\chi$ to $e_{I(\chi)}$
induces an isomorphism $Z(KA)\cong$ $Z(KB)$. In other words, since 
$Z(K'A)=$ $K'\ten_K Z(KA)$ and $Z(K'B)=$ $K'\ten_K Z(KB)$, it follows 
that the above isomorphism $Z(K'A)\cong$ $Z(K'B)$ is obtained from 
$K'$-linearly extending an isomorphism $Z(KA)\cong$ $Z(KB)$, which 
implies that this isomorphism commutes with the action of $\Aut(K'/K)$, 
whence statement (b).
\end{proof}

By Lemma \ref{signsLemma}, a virtual Morita equivalence between split 
semisimple finite-dimensional algebras given by a virtual bimodule and 
its dual is equivalent to a character bijection with signs. If the
compatibility of the character bijection with Galois automorphisms is 
all one wants to establish, one does not need to descend to valuation 
rings. For the sake of completeness, we spell this out for block 
algebras; this is an easy consequence of results of Brou\'e  
\cite{Broue90}. 

\begin{pro}\label{pro:basic}  
Let $G$ and $H$ be finite groups and let $K'/K$ be a finite Galois 
extension such that $K'$ is a splitting field for both $G$ and $H$.  
Let  $b$  be a central idempotent of $K'G$ and $c$ a central idempotent 
of $K'H$ and let $I : \Irr (K'Gb) \to \Irr(K'Hc)$  be a bijection. 
Suppose that there exist signs $\delta_{\chi} \in \{\pm 1\}$ for any
$\chi \in \Irr(K'Gb)$, such that the virtual bicharacter 
$\mu:=\sum_{\chi\in \Irr(K'Gb)} \delta_\chi (\chi\times I(\chi))$
of $G\times H$ takes values in $K$. Then $b\in KG $ and $c\in KH $.   
Moreover, the following hold.

\begin{enumerate}[\rm (a)]
\item  
For all $\sigma \in \Gal (K'/K)  $ and all  $\chi \in \Irr(K'Gb) $  
we  have $I({^{\sigma }\chi}) =\,^\sigma{I(\chi)} $. 

\item 
The  $K'$-algebra isomorphism $Z(K'Gb) \to Z(K'Hc)$ sending $e_{\chi}$ 
to $e_{I(\chi)}$ for all $\chi \in \Irr(K'Gb)$ restricts to a 
$K$-algebra automorphism  $Z(KGb) \to Z(KHc) $.
\end{enumerate}
\end{pro}

\begin{proof}
The hypothesis that $\mu$ takes values in $K$ implies that if
$\delta_\chi(\chi\times I(\chi))$ is a summand of $\mu$, then so
is $\delta_\chi({^\sigma\chi}\times {^\sigma{I(\chi)}})$. This shows 
that $b$ and $c$ are $\Gal(K'/K)$-stable, hence contained in $KG$ and 
$KH$, respectively, and it shows that $I$ commutes with the action of 
$\Gal(K'/K)$ as stated in (a). Again since $\mu$ takes values in $K$, it 
follows from the explicit formulas of the central isomorphism 
$Z(K'Gb)\cong$ $Z(K'Hc)$ in the proof of  
\cite[Th\'eor\`eme 1.5]{Broue90}  that this isomorphism restricts to 
an isomorphism $Z(KGb)\cong$ $Z(KHc)$.
\end{proof}  

 %%%%%%%%%%%%%%%%%%%%%%%%%%%%%%%%%%%%%%%%%%%%%%
 \section{Characters and  Galois automorphisms}  \label{Galois-Section}

\begin{defi}  \label{defi:pmodext}
An {\it extension of a $p$-modular system} $(K, \CO, k)$  is a  
$p$-modular system  $(K', {\CO}', k')$ such that $\CO$ is a subring of   
$\CO' $, with $J(\CO)\subseteq J({\CO}')$.  
\end{defi} 

In the situation of the above definition, we write $(K,\CO,k)\subseteq$ 
$(K',\CO',k')$, and whenever convenient, we identify without further 
notice $K$ as a subfield of $K'$ and $k$ as a subfield of $k'$ in the 
obvious way.

In this section we fix a $p$-modular system $(K, \CO, k)$ such that  
$k$ is perfect. Denote by $\bar K$ a fixed algebraic closure of $K$ and 
${\mathbb Q}_n $  the $n$-th cyclotomic extension of $\Q$   
in $\bar K$. We denote  by ${\mathcal G}_n$ the group 
$\Gal({\mathbb Q}_n/{\mathbb Q})$. The following lemma combines some
basic facts on extensions of complete discrete valuation rings; we include
proofs for the convenience of the reader. 

\begin{lem} \label{lem:stablering}  
Let $(K', {\CO}', k')$ be an extension  of  the $p$-modular system 
$(K,\CO,k)$ such that $K'$ is a normal extension of $K$.  Then,   
$\CO' $ is     $\Gal(K'/K)  $-invariant, and  $k'/k $  is a Galois 
extension. Moreover, if $\CO$ is absolutely  unramified, then  the    
homomorphism $\Gal (K'/K) \to  \Gal( k'/k) $  induced by restriction 
to  $\CO'$ is  surjective.
\end{lem}

\begin{proof}     
Let  $\pi  $  (respectively $\pi'$) be a  uniformiser of  $\CO $  
(respectively $\CO'$)  and let $a$  be  a real number  with $0 < a< 1$. 
Since  $\CO \subseteq \CO' $  and $J(\CO) \subseteq J(\CO') $, we have 
$\pi={\pi'}^e u $,  for some positive integer $e$  and some 
$u \in (\CO')^{\times} $. Let  $\nu:  K \to  {\mathbb  R} $   be  the   
absolute  value defined by $\nu ( x)  = a^{ei} $  if  $x  =\pi ^i    v$,  
$v \in \CO^{\times}$  and let $  \nu':  K' \to  {\mathbb R}$   be the 
absolute value defined by $\nu' (x)  = a^i $  if  $x  =(\pi ') ^i  v$,  
$v \in {\CO'}^{\times} $. Then  $\nu'  $, and consequently  
$\nu'\circ \sigma $ are extensions of $\nu $ to $K'$  for any 
$\sigma \in  \Gal(K'/K)  $. 
On the other hand, since   $K'$ is  an algebraic extension of $K$ and    
since  $\nu $ is complete, there is a unique extension  of $\nu $ to an 
absolute value on $K' $ (see \cite[Chapter 2, Theorem 4.8]{Neuk}).    
Thus $\nu'=  \nu'\circ \sigma  $  for all $\sigma \in  \Gal(K'/K)$. 
This proves the first assertion as  the valuation ring of  $\nu' $ is 
$\CO'$. It follows from the first assertion that  $\CO' $  is integral  
over $\CO $, and consequently that  $k'$ is  a normal extension of $k$.    
Further, since  any algebraic extension of a  perfect  field is perfect,  
$k'$ is perfect and  $k'/k  $ is separable, Hence $k'/k$ is  Galois  
as claimed.

Now suppose that  $\CO$ is absolutely unramified; that is, $\CO$ is the
ring of Witt vectors $W(k)$. Let $\CO_0= W(k') \subseteq  \CO'$ be the 
ring of Witt vectors of $k'$ in $\CO'$ and let $K_0 $ be the field of 
fractions of $\CO_0 $. Then $K_0$ is a normal extension of $K$; to see 
this it suffices to show that $K_0$ is $\Gal(K'/K)$-invariant, hence 
that $\CO_0$ is $\Gal(K'/K)$-invariant. This is obvious since $\CO_0$ 
is generated by $p$ and the canonical lift of $(k')^\times$ in 
$(\CO')^\times$, both of which are clearly $\Gal(K'/K)$-invariant.  

Applying the  first part of the lemma to the extension  
$(K_0, \CO_0, k')$  of $(K, \CO, k)$  we obtain, via restriction to   
$\CO_0 $, a homomorphism  from  $\Gal (K_0/K) \to \Gal(k'/k)$. 
This homomorphism is surjective. Indeed, by 
\cite[Chapter 2, Theorem 4]{SeLF}, any automorphism of $k'$ lifts
uniquely to an automorphism of $\CO_0$, and applying the same
theorem again shows that the unique lift of an automorphism of $k'$ 
which is the identity on $k$ is the identity on $\CO$. By 
the normality of $K'/K$  every element of $\Gal(K_0/K)$ extends to an 
element of $\Gal(K'/K)$, proving the result.
\end{proof}

\begin{lem} \label{lem:cyclotomic}   
Let $(K', {\CO}', k')$ be an extension  of  the $p$-modular system 
$(K,\CO,k)$ such that $K'$ is a normal extension of $K$ contained in 
$\bar K$. Suppose  that $\CO$ is absolutely unramified.

 \begin{enumerate}[(\rm a)]
\item  
Let $\zeta \in K'$ be a root of unity whose order is a power of $p$.  
Then $\Gal(K[\zeta]/K) \cong \Aut (\langle \zeta \rangle) \cong  
\Gal({\Q} [\zeta]/{\Q}) $.

\item    
Suppose that ${\mathbb Q}_n \subseteq K' $. Then $\CH_{n,k}$ is the 
image of the map $\Gal(K'/K) \to \CG_n$   induced by restriction to 
${\mathbb Q}_n $.
\end{enumerate}
\end{lem}

\begin{proof}   
(a) Let $m\geq 1 $ and let $\Phi_{p^m}(x) \in {\Z}[x]$  denote 
the $p^m$-th cyclotomic polynomial. We have 
$$\Phi_{p^m}(x)  =  \frac{x^{p^m}-1}{x^{p^{m-1}}-1}  = 
\Phi_p (x^{p^{m-1}})   .$$
Set $f(x) = \Phi_{p^m}(x+1)$. Then,       
$$f(x) = \Phi_p ((x+1) ^{p^{m-1}}) \equiv  \Phi_p (x^{p^{m-1}} + 1)
   \  \  \mod   \   p{\Z}[x].  $$  
Note that $\Phi_p(x+1)=$ $\frac{(x+1)^p-1}{x}=$ 
$\sum_{i=1}^p \binom{p}{i} x^{i-1}$, so all but the leading 
coefficient of this polynomial are divisible by $p$. Upon
replacing $x$ by $x^{p^m}$, it follows in particular that
all intermediate coefficients of 
$$\Phi_p (x^{p^{m-1}} + 1) = 
\frac{(x^{p^{m-1}} + 1)^p-1}{x^{p^{m-1} }}$$
are divisible by $p$. Thus all intermediate coefficients of $f(x)$  
are divisible by $p$. Also, $f(x)$ is monic and has constant term 
$p$. Since $p$ is prime in $\CO$, it follows by Eisenstein's 
criterion  applied to $\CO$, that $\Phi_{p^m} (x)$ is irreducible 
in $\CO[x]$, and hence by Gauss's lemma  that $\Phi_{p^m}(x) $ is 
irreducible  in $K[x]$. This proves the first assertion. 

(b) We first show that the image of $\Gal(K'/K)$ in $\CG_n$ is contained 
in $\CH_{n,k} $. By  Lemma~\ref{lem:stablering}, restriction to $\CO$ 
induces a homomorphism $\Gal(K'/K) \to  \Gal(k'/k)$. Let 
$\tau \in \Gal(K'/K)  $  and denote  by $\bar \tau  $ the image of  
$\tau \in  \Gal(k'/k)$  under the above map.  The restriction   of 
$\bar \tau $ to    the  (finite) splitting field   of $ x^n -1 $ over 
${\mathbb F}_p$ is a power of the Frobenius map $ x \to x^p $.   
Since the canonical  surjection  $u \to \bar u$  from $\CO '$ to $k'$   
induces an isomorphism between the groups of $p'$-roots of unity of $K'$  
and of $k'$ and  since $\bar \tau  $  is the identity on $k$, it follows 
that the restriction of $\tau  $ to ${\mathbb Q}_n  $ is  an element of 
$\CH_{n,k} $.   

Next we show that $\CH_{n,k}   $ is contained  in  the image of 
$\Gal(K'/K)$ in $\CG_n$. Let  $\zeta\in\Q_n$ be a primitive  $n$-th 
root of unity. Write $\zeta= \zeta_p \zeta_{p'} $, where 
$\zeta_p, \zeta_{p'}$ are powers of $\zeta$, the order of $\zeta_p$ is 
a power of $p$, and the order of $\zeta_{p'}$ is prime to $p$.    
Let $\alpha \in \CH_{n,k} $.    We will  prove  that  there exists 
$\beta \in\Gal(K'/K) $ such that the restriction of $\beta $ to $\Q_n $  
equals $\alpha $.       By  Lemma~\ref{lem:stablering},   $k' $  is  a 
normal extension of $k$. Hence by the definition of  $\CH_{n,k} $, there 
exists   $\bar \tau \in \Gal(k'/k) $   and   a non-negative integer  $u$   
such that $\alpha (\delta) =$ $\delta^{p^u}$   for all $n_{p'}$-roots of 
unity $\delta$ in $\Q_n$  and $\bar \tau( \eta)=$ $\eta^{p^u}$  for all 
$n_{p'} $-roots of  unity  $\eta $ in $k'$. Again by 
Lemma ~\ref{lem:stablering}, $\bar \tau $ lifts to an automorphism 
$\tau \in \Gal(K'/K)  $.  By the isomorphism between the groups of 
$p'$-roots of unity in $K'$ and in $k $, we have that  
$\tau (\zeta_{p'} )=$ $\alpha(\zeta_{p'}) $. By part (a), there exists 
$\sigma\in$ $\Gal(K_0[\zeta_p]/K_0) $  such that 
$ \sigma (\tau( \zeta_p) ) =$ $\alpha (\zeta_p)$.  Let $\sigma'$ be any 
extension of $\sigma$ to $K' $ and set  $\beta:=$ $\sigma'\tau$.  Then 
$\beta $ has the required properties.
\end{proof}

Theorem~\ref{thm:intro-basic} forms the first part of the statement of 
the following result. For a finite group $G$ denote by $\IBr(G)$ the set 
of irreducible Brauer characters of $G$ interpreted as functions from 
the set of $p$-regular elements of $G$ to $\bar K$.  If $b$ is a central 
idempotent of $\CO G$, then we denote by $\IBr(G,b) $ the subset of 
$\IBr(G)$ consisting of the Brauer characters of simple $k'Gb $-modules
for  any sufficiently large field  $k'$ containing $k$.

\begin{thm} \label{thm:rational} 
Let $G, H$ be finite groups and let $n$ be large enough for $G$ and $H$.
Let $b$, $c$ be blocks of $\CO G$ and  $\CO H$ respectively.  Suppose 
that $\CO $ is  absolutely unramified. Any virtual Morita equivalence  
between $\CO Gb $ and $\CO Hc $ given by a virtual bimodule and its 
dual induces a bijection $I :  \Irr(\bar K Gb) \to \Irr(\bar K Hc)$ 
satisfying  $\,^{\sigma}I(\chi) = I(\,^{\sigma} \chi) $ for all 
$\sigma \in {\mathcal H}_{n,k} $ and all  $\chi \in Irr(\bar K Hc)$. 

If the virtual Morita equivalence is induced from a Morita equivalence,
then  in addition there exists a  bijection, $\bar I:  \IBr(G, b) \to$
$\IBr(H, c) $  such that for  all $\chi \in \Irr(G,b) $, 
$\varphi \in \IBr(G,b)$ and $\sigma \in {\mathcal H}_{n, k}$ the 
decomposition numbers  $d_{I(\chi),  \bar I (\varphi)}$ and 
$d_{\chi,   \varphi}$ are equal  and $\, ^{\sigma} \bar I (\varphi)=$
$\bar I (\,^{\sigma} \varphi ) $.
\end{thm}

\begin{proof}    
Let $(K',\CO', k')$ be  an extension of  the $p$-modular system
$(K,\CO,k)$ such that $K'\subseteq$ $\bar K$ and such that the 
extension $K'/K$ is normal. Suppose that  $k'$ is perfect, and that 
$K'$  contains primitive  $n$-th, $|G|$-th and $|H|$-th roots of 
unity.  We may  and will  identify $\Irr(\bar KG) $ and $\Irr(\bar KH)$ 
with  $\Irr (K'G)$ and $\Irr(K'H)$ respectively.  By 
Lemma~\ref{lem:cyclotomic}, the subgroup $\CH_{n,k} $ is the image of 
the restriction map from $\Gal(K'/K)$ to $\CG_n$. It follows from 
\ref{thm:introvirtual} that a virtual Morita equivalence between
$\CO Gb$ and $\CO Hc$ given by a virtual bimodule and its dual yields
a character bijection $I$ which commutes with $\Gal(K'/K)$, hence with
$\CH_{n,k}$. By \cite[1.2]{Broue90}, the bijection $I$, together with 
the signs from \ref{signsLemma}  is a perfect isometry. 
  
Now suppose that $X$ is a $(\CO Hc, \CO Gb)$-bimodule finitely 
generated and  projective as left $\CO Hc $-module and  as right 
$\CO Gb$-module such that $X\otimes_{\CO Gb}-$ induces a Morita  
equivalence  between $\CO Gb $ and $\CO Hc $. Then $X':=\CO' \tenO X$ 
induces a Morita equivalence between $\CO'Gb $ and $\CO' Hc$ and 
$\bar X := k\otimes_{\CO} X $ and $\bar X':=  k'\otimes_{k} \bar X$ 
induce  Morita equivalences between $kGb$ and $kHc $ and between $k'Gb$ 
and $k'Hc$ respectively. 

Since $k'$ contains  enough roots of unity, we may identify $\IBr(G,b) $  
(respectively $\IBr(H,c) $) with the Brauer characters  of simple 
$k'Gb $-modules (respectively $k'Hc$-modules). For a simple 
$k'Gb$-module (or $k'Hc$-module)  $S$, denote by $\varphi_S$ the  
corresponding Brauer character. Let $\bar I: \IBr(G, b) \to$
$\IBr( H, c)$ be the bijection induced by $\bar X'$, that is such 
that $\bar I (\varphi_S)  =$ $\varphi_{\bar  X' \otimes _k' S}$ for any  
simple $k'Gb$-module  $S$.  Since $\bar X' \cong$
$k'\otimes_{{\CO}'} X' $, we have that 
$d_{I(\chi), \bar I (\varphi) }= d_{\chi, \varphi }$ for all 
$\chi \in \Irr_{K'}(G, b)$, $\varphi \in IBr(b)$.

Let $\sigma \in {\mathcal H}_{n,k}$. By the previous lemma, there 
exists $\tau  \in    \Gal(K'/K)$  such that the restriction to 
${\mathbb Q}_n $    is $\sigma $. Let $\bar \tau  \in \Gal(k'/k)$ be 
the image of $\tau $.    If  $S$ is any  simple $k'Gb$-module 
(respectively $k'Hc $-module),  then $\, ^{\sigma} (\varphi_S) =$
$\varphi_{\,^{\bar \tau} S}$. Thus, it suffices to show that 
$\, ^{\bar\tau}(\bar X' \otimes_{k'}S) \cong$
$\bar X' \otimes_{k'} \,^{\bar \tau} S $ for any simple $k'Gb$-module 
$S$.
Now $$\,^{\bar \tau} X'=\,^{\bar \tau }(k'\otimes_k X)\cong k'\otimes X$$ 
as $(k'Hc, k'Gb)$-bimodule and it follows that for any simple 
$k'Gb $-module $S$  that 
$$ \,^{\bar\tau} (X'  \otimes_{k'Hc} S) \cong   
\,^{\bar\tau} X' \otimes_{k'Hc} \,^{\bar\tau} S  \cong  
X' \otimes_{k'Hc}  \,^{\bar\tau}  S  $$ 
as $k'Hc$-modules.  This proves the result.
\end{proof}

\begin{proof} [{Proof of Corollary~\ref{cor:intro-rational}}]  
Let  $n$  be a common multiple of $|G|$ and of $|H|$. For any $g \in G$  
and $\chi \in \Irr(\bar K G)$, $\chi(g) \in \Q_n $.  By the basic theory  
of cyclotomic extensions  of $\Q$,   $\chi $ is $p$-rational if and only 
if $\, ^{\sigma} \chi =\chi $ for all  $\sigma \in \CG_n$ such that 
$\,^\sigma  (\eta)  =\eta$ for all $n_{p'}$-roots of unity $\eta \in \Q_n$ 
and  similarly  for the characters of $H$.  
On the other hand, if $\sigma \in \CG_n$ is such that 
$\,^\sigma  (\eta) =\eta $ for all $n_{p'}$-roots of unity $\eta \in \Q_n$, 
then $\sigma \in \CH_{n,k} $. The result is now immediate from 
Theorem~\ref{thm:rational}.
\end{proof}

\begin{proof} [{Proof of Corollary~\ref{cor:intro-virtual}}]   
The  action of $\CH_n$ on $\Irr(\bar K G)$ induces an action of 
$\CH_n $  on the set of blocks of $\CO G$.   Since $\CO$ is a minimal 
complete discrete valuation ring for $b$,   $k$  is  a finite field  
and consequently  a splitting field of $x^{n_{p'}}-1 $ over $k $ is also 
finite.   Let $ |k|= p^d $. Then  $\CH_{n,k} $   consists of  precisely 
those elements $\alpha $ of $\CG_n $  for which there exists a 
non-negative integer $u$ such that  $\sigma(\delta)=$ $\delta^{p^{ud}}$  
for all $n_{p'} $-roots of unity  in ${\mathbb Q}_n$. It follows that 
$\CH_{n,k}$ is the stabiliser of  $b$  in $\CH_n$  and similarly for 
$H$ and $c$. The result is now immediate  from Theorem~\ref{thm:rational}.
\end{proof}
`
%%%%%%%%%%%%%%%%%%%%%%%%%%%%%%%%%%%%%%%%%%%%%%%%%
\section{Descent  for  equivalences} \label{descentequiv-Section}

Let $\CO$, $\CO'$ be complete local commutative principal ideal  
domains such that
$\CO\subseteq$ $\CO'$ and $J(\CO)\subseteq$ $\CO'$ (so that either
$\CO$, $\CO'$ are complete discrete valuation rings or they are fields, 
allowing the possibility that $\CO$ is a field but $\CO'$ is not). Let 
$A$ be an $\CO$-algebra which is finitely generated as an $\CO$-module.
Then $A$ is in particular noetherian, and hence the category of
finitely generated $A$-modules $\mod$-$A$ is abelian and coincides with 
the category of finitely presented $A$-modules. We set $A'=$ 
$\CO'\tenO A$.
For any $A$-module $U$ we denote by $U'$ the $A'$-module $\CO'\tenO U$,
and for any homomorphism of $A$-modules $f : U\to V$, we denote
by $f'$ the homomorphism of $A'$-modules $\Id_{\CO'}\ten f : U'\to$
$V'$. We extend this notion in the obvious way to complexes; that is,
if $X=$ $(X_n)_{n\in\Z}$ is a complex of $A$-modules with differential 
$\delta=$ $(\delta_n)_{n\in\Z}$, then we denote by $X'$ the complex of 
$A'$-modules $(X'_n)_{n\in\Z}$ with differential $\delta'=$ 
$(\delta'_n)_{n\in\Z}$.

The following lemmas are adaptations to the situation 
considered in this paper of well-known results which hold in greater 
generality; see the references given. We do not require the ring $\CO'$ 
to be finitely generated as an $\CO$-module.

\begin{lem} \label{Offlat}
The ring extension $\CO\subseteq$ $\CO'$ is faithfully flat; that is,
the functor $\CO'\tenO - $ is exact and sends any nonzero $\CO$-module 
to a nonzero $\CO'$-module.
\end{lem}

\begin{proof}
Since $\CO'$ is torsion free as an $\CO$-module, it follows from 
\cite[(4.69)]{Lam} that $\CO'$ is flat as an $\CO$-module. 
Since $J(\CO)=$ $\CO\cap J(\CO')$ is the unique prime ideal in
$\CO$, it follows from \cite[(4.74)]{Lam} or also \cite[(4.71)]{Lam}
that $\CO'$ is faithfully flat as an $\CO$-module.
\end{proof}

\begin{lem}\label{fflat}    
Let $A$ be an $\CO$-algebra which is finitely generated as an 
$\CO$-module. Then for any sequence 
$\xymatrix{M\ar[r]^{f} & N\ar[r]^{g} & U}$
of $A$-modules, 
the sequence is exact at $N$ if and only if the sequence
$\xymatrix{M'\ar[r]^{f'} & N'\ar[r]^{g'} & U'}$
of $A'$-modules is exact at $N'$.
\end{lem}

\begin{proof} 
This follows from Lemma \ref{Offlat} and \cite[(4.70)]{Lam}.
\end{proof}

Recall that a morphism   $\alpha :  X \to Y $ in a category  
${\mathcal C}$  is  {\it split} if there exists a morphism 
$\beta :  Y \to X $  such that $\alpha\beta\alpha =\alpha $.
     
\begin{lem}\label{split} 
Let $A$ be an $\CO$-algebra which is finitely generated as an
$\CO$-module. A morphism $f: M \to N$ in $\mod$-$A$ is split if and 
only if the morphism $f': M'\to N'$ in $\mod$-$A'$ is split.
\end{lem}  

\begin{proof}   
One checks easily that $f$ is split if and only if the two
epimorphisms $M\to$ $\Im(f)$ and $N\to$ $\coker(f)$ are split.
 By Lemma \ref{Offlat}, the extension $\CO\subseteq$ $\CO'$ 
 is faithfully flat, and hence $\Im(f')=$ $\Im(f)'$ and
 $\coker(f')=$ $\coker(f)'$. 
Thus it suffices to show that $f$ is a split epimorphism if and
only if $f'$ is a split epimorphism. 
Now $f$ is a split epimorphism if and only of the map
$\Hom_A(N,M)\to$ $\Hom_A(N,N)$ induced by composing with $f$ is 
surjective (since in that case an inverse image of $\Id_N$ under
this map is a section of $f$). Again since the
extension $\CO\subseteq$ $\CO'$ is faithfully flat, it follows that
$f$ is a split epimorphism if and only if the induced map
$\CO'\tenO \Hom_A(N,M)\to$ $\CO'\tenO \Hom_A(N,N)$ is surjective.
Since $M$, $N$ are finitely generated, hence finitely presented
by the assumptions on $A$, it follows from \cite[Theorem I.11.7]{NaTs}
(applied with $\CO'$ instead of $B$) that there is a canonical 
isomorphism $\CO'\tenO \Hom_A(N,M)\cong$ $\Hom_{A'}(N',M')$ and a 
similar isomorphism with $N$ instead of $M$.
Thus the surjectivity of the previous map is equivalent to 
the surjectivity of the map $\Hom_{A'}(N',M')\to$ $\Hom_{A'}(N',N')$
induced by composing with $f'$. This is, in turn, equivalent to
asserting that $f'$ is a split epimorphism, whence the result.
\end{proof}

\begin{lem} \label{contract}  
Let $A$ be an $\CO$-algebra which is finitely generated as an 
$\CO$-module, let $X$ be a complex of finitely generated $A$-modules  
and let $M$ be a finitely generated $A$-module. Then   

\begin{enumerate} [\rm(a)] 

\item  
$M$ is projective  if and only if $M'$ is  projective as $A'$-module.

\item 
$X$ is acyclic if and only if $X'$ is acyclic as complex of 
$A'$-modules.

\item   
$X$ is contractible if and only if $X'$ is contractible as complex of 
$A'$-modules. 
\end{enumerate} 
\end{lem}

\begin{proof}    
Part (a) and part (b) follow  from Lemma \ref{fflat} and Lemma 
\ref{split}. By \cite[\S 1.4]{Weibel}, a complex of $A$-modules 
(respectively $A'$-modules) is contractible if and only if the complex 
is acyclic and the differential in each degree is split. Therefore 
part (c) also follows from  Lemma~\ref{fflat} and Lemma~\ref{split}.
\end{proof}

Let $A$ and $B$ be symmetric $\CO$-algebras. 
Let $M$ be a finitely generated  $(A,B)$-bimodule 
which is projective as left $A$-module and  as right $B$-module. 
If $M\tenB - :\modh B\to$ $\modh A$ is an equivalence, then the 
symmetry of $A$ and $B$ implies that an inverse of this 
equivalence is induced by tensoring with the $\CO$-dual 
$M^\vee$ of $M$; that is, $M\otimes_B M^\vee \cong A$ as 
$(A, A)$-bimodules and  $ M^\vee \otimes_A M\cong B $ as 
$(B, B)$-bimodules.  Following Brou\'e, we say that $M$ induces a {\it 
stable equivalence  of Morita type} if there exist a projective 
$(A, A)$-bimodule $U$ and a projective $(B, B)$-bimodule $V$ such that 
$M \otimes_B M^\vee \cong A \oplus U$ as $(A, A)$-bimodules  
and $ M^\vee \otimes_A M\cong B  \oplus V  $ as $(B, B)$-bimodules. 
Let $X$ be a bounded complex of  finitely generated  $(A,B)$-bimodules 
which are projective as left $A$-modules   and  as right $B$-modules,  
and let $X^\vee=\Hom_{\CO} (X, \CO)$ be the dual complex. We say that 
$X$ induces a {\it Rickard  equivalence} and that $X$ is a {\it Rickard 
complex} if  there exist a contractible complex of $(A, A)$-bimodules 
$Y$ and a contractible  complex of $(B, B)$-bimodules $Z$ such that 
$X \otimes_B X^\vee \cong$ $A \oplus Y$ as complexes $(A, A)$-bimodules 
and $X^\vee \otimes_A X\cong$ $B \oplus Z$ as complexes of 
$(B, B)$-bimodules. Let $M$ and $N$ be finitely generated  
$(A, B)$-bimodules,   projective as left and  right  modules and let 
$U=[M]-[N]$. Then $U^\vee=[M^\vee]-[N^\vee]$. Recall that  $U$ and 
$U^\vee$  induce a   virtual Morita equivalence  between 
$A$ and $B$   if   $U \cdot_B\  U^\vee = $ $[A]$  in $\CR(A,A)$
and  $U^\vee \cdot_B\  U = $   $[B]$  in $\CR(B,B)$. 
We denote by $C^b(A)$ the category of bounded complexes of finitely 
generated $A$-modules, by $K^b (A)$  the homotopy category of bounded 
complexes  of  finitely generated $A$-modules and by $D^b(A)$ the  
bounded derived category of $\modh A$.  For a  finitely generated  
$A$-module  $M$ we denote  by $[M] $ the  isomorphism class of $M$  
as an element of the Grothendieck group of $\modh A$ with respect to 
split exact sequences. We use the analogous notation for bimodules.

\begin{pro} \label{extendProp}
Let  $A$ and $B$ be symmetric $\CO$-algebras. Let $M$, $N$  be  
finitely generated $(A,B)$-bimodules which are  projective as left 
$A$-module  and as right $B$-module  and let $X$ be a bounded complex 
of  finitely generated $(A,B)$-bimodules which are  projective as left 
$A$-modules and as right $B$-modules. 

\begin{enumerate} [\rm (a)]

\item  
$X'$ induces a Rickard equivalence between $A'$ and $B'$, 
if and only if  $X$ induces a Rickard equivalence between $A$ and $B$.

\item  
$M'$ induces a stable equivalence of Morita type  between $A'$ and $B'$  
if and only if  $M$ induces a stable equivalence of Morita type between 
$A$ and $B$.

\item  
$M'$ induces a Morita equivalence between $A'$ and $B'$ if and only if 
$M$ induces a Morita equivalence between $A$ and $B$.
 
\item  
$[M']-[N']$  and $[(M')^\vee]-[(N')^\vee]$ induce  a virtual Morita 
equivalence between $A'$ and $B'$  if and only if $[M]-[N]$ and 
$[M^\vee]-[N^\vee]$ induce a virtual Morita equivalence between $A$ and 
$B$.
\end{enumerate}
\end{pro} 

\begin{proof}   
One direction of the  implications is  trivial. We verify the  reverse 
implications. We will apply the previous lemmas in this section to 
the $\CO$-algebras $A \tenO B^{\op} $,  $A\tenO A^{\op} $ 
etc. In what follows, we will freely switch between the terminology 
of $(A, A)$-bimodules and $A \tenO A^{\op} $-modules.

We prove (a).   
Suppose that $X'$ induces a Rickard equivalence between $A'$ and 
$B'$ and let $Y$ be a  contractible bounded  complex of 
$(A', A')$-bimodules such that  $ X' \otimes_{B'} (X')^\vee  =$ 
$A' \oplus Y $.

The functors  $X\otimes_{B} - $ and $X^\vee\otimes_A- $ define a pair of  
biadjoint functors between $C^b(B\otimes _{\CO}A ^{\op})$ and   
$C^b(A\otimes_{\CO} A^{\op} ) $ (this is well-known; see e. g. 
\cite[Section 6.10]{Lintransfer}). Denote by 
$\epsilon_X: X \otimes_B X^\vee \to$ $A$ and
$\epsilon_{X^\vee}: X^\vee\otimes_A X \to$ $B$ the counits of these  
adjunctions. Similarly, denote by  
$\epsilon_{X'}: X' \otimes_{B'} (X')^\vee \to A'$ and
$\epsilon_{(X')^\vee}: (X')^\vee\otimes_{A'} X' \to B$     
the counits corresponding to the biadjoint pair 
$ X'\otimes_{B'} - $ and $(X')^\vee \otimes _{A'} - $ between 
$C^b(B'\otimes_{\CO} (A')^{\op}) $ and 
$C^b(A'  \otimes_{\CO} (A')^{\op})$.
Since  the terms of $X$ are finitely generated and $\CO$-free, we have 
that $(X')^\vee \cong (X^\vee)'$ . Hence we may assume that   
$\epsilon_{X'} =\epsilon_{X} ' $ and 
$\epsilon_{(X')^\vee} =\epsilon_{X^\vee}'$. Now the hypothesis implies 
that $X' \otimes_{B'}- $ and $(X')^\vee\otimes_{A'} - $  define a pair 
of inverse equivalences between $K^b(A' \otimes_{\CO'} (A')^{\op})$ and  
$K^b(B' \otimes_{\CO'} (A')^{\op} )$.  Thus  
$\epsilon_{X'} : X' \otimes_{B'} (X')^\vee \to A'$ is an 
isomorphism in $K^b(A'\otimes_{\CO} (A')^\op)$ (see for instance  
\cite[Chapter 4, \S 2, Prop.~4]{Mac}). Since $A'$ is concentrated    
in a single degree it follows that $\epsilon_X$ is split surjective  
in  $C^b(A'\otimes_{\CO'} (A')^{\op}) $ and 
$X' \otimes_{B'} (X')^\vee  =$ $A' \oplus \Ker(\epsilon_{X'})$ in 
$C^b(A'\otimes_{\CO'} (A')^{\op})$. Since we also have 
$X' \otimes_{B'} (X')^\vee  =$ $A' \oplus Y$  in  
$C^b(A'\otimes_{\CO'} (A')^{\op})$  with $Y$ contractible, by the 
Krull-Schmidt property of $C^b(A'\otimes_{\CO'} (A')^{\op})$ we have 
that $\Ker(\epsilon_{X'})$ is contractible. By Lemma~\ref{fflat} we have 
that $\Ker(\epsilon_{X'}) =$ $\Ker(\epsilon_{X}') =$
$(\Ker(\epsilon_{X}))'$. Hence by Lemma \ref{contract} we have 
that $\Ker(\epsilon_{X}) $ is contractible as a complex of 
$(A,A)$-bimodules. Similarly by Lemma~\ref{fflat} we have that  
$\epsilon_X$ is surjective and  by Lemma \ref{split} that  
$\epsilon_X$ is split (note that since $A$ is concentrated in a  
single degree, namely zero, the split surjectivity of 
$\epsilon_{X}$ as map of complexes is equivalent to the split 
surjectivity of the degree $0$-component of $\epsilon_{X} $). Thus, 
we have that $X \otimes_{B} {X}^\vee = $ $A \oplus \Ker(\epsilon_X)$ 
as complexes of $(A, A)$-bimodules and $\Ker(\epsilon_X)$ is  
contractible. Arguing similarly for $X^\vee\otimes_{A} {X}$ proves (a).

The proof of  (b)  proceeds along the same lines  as that of (a), the 
contractibility arguments are replaced by the fact that if  $U$ is a  
finitely generated $(A,A)$-bimodule, then $U$ is projective if $U'$ 
is a projective $(A', A')$-bimodule  (Lemma~\ref{contract}).
The proof of  (c)  is a special case of the proof of (b).   

Statement (d) is a consequence of a version, due to Grothendieck, of 
the  Noether-Deuring Theorem for the base rings under consideration. 
More precisely,  if $[M]-[N]$ and its dual induce a virtual Morita  
equivalence, then 
$$[M\tenB M^\vee]+[N\tenB N^\vee] - 
[M\tenB N^\vee]-[N\tenB M^\vee]=[A]\ ,$$ 
which is equivalent to the existence of an $(A,A)$-bimodule isomorphism 
$$M\tenB M^\vee\oplus N\tenB N^\vee\cong 
A\oplus M\tenB N^\vee\oplus N\tenB M^\vee\ .$$ 
By \cite[Proposition (2.5.8) (i)]{GrothEGA4} such an isomorphism exists 
if and only if there exists an analogous isomorphism for the 
corresponding $(A',A')$-bimodules, whence (d). 
\end{proof}

\begin{rem} 
While the categorical equivalences in (a), (b), (c) in the theorem above 
induced by a bimodule or a complex of bimodules have the property that
their inverses are automatically induced by the dual of that bimodule
or complex, this is not true for virtual Morita equivalences, whence the
extra hypothesis in (d). For instance, if $A$ is a split semisimple 
$K$-algebra with $m$ isomorphism classes of simple modules, then any 
matrix $(a_{i,j})$ in $\SL_m(\Z)$ with inverse $(b_{i,j})$ yields a 
virtual self Morita equivalence of $A$ of the form 
$\sum_{i,j} a_{i,j} [S_i\tenk S_j^\vee]$ with inverse
$\sum_{i,j} b_{i,j} [S_i\tenk S_j^\vee]$, where $\{S_i\}$ is a set of 
representatives of the isomorphism classes of simple $A$-modules, and 
where the indices $i$, $j$ run from $1$ to $m$. 
\end{rem}

%%%%%%%%%%%%%%%%%%%%%%%%%%%%%%%%%%%%%%%%%%%
\section{Descent and relative projectivity} 
\label{descentrelproj-Section}

Let $(K,\CO,k)\subseteq$ $(K',\CO',k')$ be an extension of $p$-modular 
systems (see Definition \ref{defi:pmodext}).
Let $G$ be a finite group and $P$ a subgroup of $G$. An $\OG$-module 
$U$ is called {\it relatively $P$-projective}, if $U$ is isomorphic to 
a direct summand of $\Ind^G_P(V)=$ $\OG\tenOP V$ for some $\OP$-module 
$V$, where $\OG$ is regarded as an $\OG$-$\OP$-bimodule. Dually, $U$ 
is {\it relatively $P$-injective}, if $U$ is isomorphic to a direct 
summand of $\Hom_\OP(\OG, V)$ for some $\OP$-module $V$, where 
$\OG$ is regarded as an $\OP$-$\OG$-bimodule. It is well-known that
because $\OG$ is symmetric, the notions of relative projectivity and
relative injectivity coincide.  Any $\OG$-module is relatively 
$\CO P$-projective, where $P$ is a Sylow $p$-subgroup of $G$.  Following 
Green \cite{Greenindec}, a {\it vertex  of a finitely generated 
indecomposable $\OG$-module $U$} is a minimal $p$-subgroup $P$ of $G$ 
such that $U$ is relatively $P$-projective. In that case, $U$ is 
isomorphic to a direct summand of $\Ind^G_P(V)$ for some finitely 
generated indecomposable $\OP$-module $V$, called an {\it $\OP$-source 
of $V$}, and then $V$ is isomorphic to a direct summand of 
$\Res^G_P(U)$.  If $P$ is clear from the context, $V$ is just called 
a source of $U$. The vertex-source pairs $(P,V)$ of $U$ are unique up 
to $G$-conjugacy. See e. g. \cite[\S 18]{Thev} for details.

\begin{lem}[{\cite[(III.4.14)]{Feitbook}}] \label{lemma-c1}
Suppose that $\CO'$ is finitely generated as an $\CO$-module.
Let $G$ be a finite group and $U$ a finitely generated  $\CO$-free 
indecomposable $\OG$-module. Let $P$ be a vertex of $U$. Then $P$ is
a vertex of every indecomposable direct summand of the $\CO'G$-module
$\CO' \tenO U$.
\end{lem} 

\begin{lem} \label{lemma-source}
Suppose that $\CO'$ is finitely generated as an $\CO$-module. Let $G$ 
be a finite group and $U$ a finitely generated  $\CO$-free 
indecomposable $\OG$-module. Let $P$ be a vertex of $U$. Let $V$ be an 
indecomposable direct summand of the $\CO'G$-module $U'=$ $\CO'\tenO U$ 
and let $Y$ be an $\CO'P$-source of $V$. Suppose that $Y\cong$ 
$\CO'\tenO X$ for some $\OP$-module $X$. Then $X$ is an $\OP$-source of 
$U$, and every indecomposable direct summand of $U'$ has $Y$ as a 
source. In particular, $U$ has a trivial source if and only if every 
indecomposable direct summand of $U'$ has a trivial source.
\end{lem}

\begin{proof} 
This is basically a special case of the Noether-Deuring Theorem; we 
sketch the argument. Since $\CO'$ is finitely generated as an 
$\CO$-module, we have $\CO'\cong$ $\CO^d$ for some positive integer $d$. 
Thus restricting $U'$ to $\OG$ yields an $\OG$-isomorphism $U'\cong$ 
$U^d$, and hence, as an $\OG$-module, $V'$ is isomorphic to $U^{c}$ for 
some positive integer $c$, by the Krull-Schmidt Theorem.  Similarly, we 
have an $\OP$-isomorphism $Y\cong$ $X^d$. Since $Y$ is isomorphic to
a direct summand of $\Res^G_P(V)$, it follows again from the 
Krull-Schmidt Theorem that $X$ is isomorphic to a direct summand of 
$\Res^G_P(U)$. By Lemma \ref{lemma-c1}, $P$ is a vertex of $X$ and of 
$U$, and therefore $X$ is a source of $U$. Since $U$ is isomorphic to a 
direct summand of $\Ind^G_P(X)$, it follows that $U'$ is isomorphic to 
a direct summand of $\Ind^G_P(Y)$. This implies that every 
indecomposable summand of $U'$ has $Y$ as a source. The last statement 
follows from the special case where $Y=$ $\CO'$.
\end{proof}

We use the following concepts and
results from Kn\"orr \cite{Knoerr78} and Th\'evenaz \cite{Thev85}.
Let $G$ be a finite group, $P$ a $p$-subgroup of $G$ and $U$ a
finitely generated $\OG$-module. 
A {\it relative $P$-projective presentation of $U$} is a 
pair $(Y,\pi)$ consisting of a relatively $P$-projective $\OG$-module 
$Y$ and a surjective $\OG$-homomorphism $\pi : Y\to$ $U$ whose 
restriction to $\OP$ is split surjective.
Such a presentation is called a {\it relatively $P$-projective cover} 
if in addition $\ker(\pi)$ has no nonzero relatively $P$-projective
direct summand; by \cite[Cor. (1.9)]{Thev85} this is equivalent to
$\pi$ being {\it essential}; that is, any endomorphism $\beta$ of
$Y$ satisfying $\pi=$ $\pi\circ\beta$ is an automorphism of $Y$.
The results in \cite{Knoerr78} and \cite{Thev85} imply that $U$
has a relative projective resolution which is unique up to 
isomorphism and which is additive in $U$. Moreover, if $U$ is
indecomposable and not relatively $P$-projective, and if $(Y,\pi)$ is
a relatively $P$-projective cover of $U$, then $\ker(\pi)$ is
indecomposable and not relatively $P$-projective. These results, 
together with Lemma \ref{lemma-c1}, imply immediately the 
following.

\begin{lem} \label{lemma-c3}
Suppose that $\CO'$ is finitely generated as an $\CO$-module.
Let $G$ be a finite group and $R$ a $p$-subgroup of $G$. Let $U$ be an 
$\CO$-free $\OG$-module which has no nonzero relatively
$R$-projective direct summand. Let $(Y,\pi)$ be a relatively
$R$-projective cover of $U$. Then the $\CO'G$-module $U'=$ 
$\CO'\tenO U$ has no nonzero relatively $R$-projective direct summand
and $(Y',\pi')=$ $(\CO'\tenO Y, \Id_{\CO'}\ten \pi)$ is
a relatively $R$-projective cover of $U'$.
\end{lem}

%%%%%%%%%%%%%%%%%%%%%%%%%%%%%%%%%%%%%%%%
\section{Descent and Galois automorphisms} \label{descentGalois-Section}

Let $(K,\CO,k)\subseteq$ $(K',\CO',k')$ be an extension of $p$-modular 
systems (see Definition~\ref {defi:pmodext}).
The following Lemma, due to Reiner, makes use of the fact that finite 
fields have trivial Schur indices.

\begin{lem}[{\cite[Theorem 3]{Re66}, \cite[(30.33)]{CRI}}] 
\label{lemma-c2}
Suppose that the field $k$ is finite. Let $G$ be a finite group and $U$ 
a finitely generated $\CO$-free indecomposable $\OG$-module. Then the 
indecomposable direct summands in a decomposition of $\CO'\tenO U$ as 
an $\CO'G$-module are pairwise nonisomorphic.
\end{lem}

Denote by $\Gamma$ the automorphism group of the field extension $k'/k$.
Let $A$ be a finite-dimensional $k$-algebra, set $A'=$
$k'\tenk A$, and let $U'$ be an $A'$-module.
We say that $U'$ is {\it $\Gamma$-stable}, if $U'\cong$ 
${\,^\sigma{U'}}$ as $A'$-modules, for all $\sigma\in$ $\Gamma$. If 
$U'\cong$ $k'\tenk U$ for some $A$-module, then $U'$ is $\Gamma$-stable. 
Indeed, the map sending $\lambda\ten u$ to $\sigma^{-1}(\lambda)\ten u$ 
is an isomorphism $k'\tenk U\cong$ ${\,^\sigma(k'\tenk U)}$, where 
$\sigma\in$ $\Gamma$, $u\in$ $U$, and $\lambda\in$ $k'$. The following 
is well-known.

\begin{lem} \label{lemma-c4}
Suppose that the fields $k'$ and $k$ are finite. Let $A$ be a 
finite-dimensional $k$-algebra. Set $A'=$ $k'\tenk A$. Suppose 
that the semisimple quotient $A/J(A)$ is separable.
Let $\Gamma$ be the Galois group of the extension $k'/k$.

\begin{enumerate} [\rm(a)] 

\item
Let $S$ be a simple $A$-module. Then the $A'$-module $S'=$
$k'\tenk S$ is semisimple, isomorphic to direct sum of pairwise
nonisomorphic Galois conjugates of a simple $A'$-module $T$.

\item  
Let $S'$ be a semisimple $A'$-module. There exists a semisimple
$A$-module $S$ satisfying $S'\cong$ $k'\tenk S$ if and only if
$S'$ is $\Gamma$-stable.

\item  
Let $Y'$ be a finitely generated projective $A'$-module. There exists 
a projective $A$-module $Y$ satisfying $Y'\cong$ $k'\tenk Y$ if and only 
if $Y'$ is $\Gamma$-stable. 
\end{enumerate}
\end{lem}

For the remainder of this section, assume that $k$, $k'$ are finite and 
that $\CO$, $\CO'$ are  absolutely unramified. Set $d=$ $[k':k]$. Then 
$\CO'$ is free of rank $d$ as an $\CO$-module. Let $\sigma : k'\to k'$ 
be a generator of $\Gal(k'/k) $.  Denote  by the same letter  
$\sigma: \CO' \to\CO'$ the unique ring automorphism of $\CO'$ 
lifting $\sigma$. 

Let $A$ be an $\CO$-algebra which is free of finite rank as an 
$\CO$-module. Set $A'=$ $\CO'\tenO A$. Let $\tau: \modh A'\to$   
$\modh A' $ be the functor which sends an $A'$-module $U$ to  the 
$A'$-module $\tau(U):=  \oplus_{0\leq i\leq d-1} \,^{\sigma^{i}} U$ 
and   a morphism $f: U\to V$ of $A'$-modules to the morphism   
$\tau (f) :=(f,  \ldots,  f ) $. Let $\delta: \modh A'$ $\to \modh A$   
be the  functor  which sends an $A'$-module  $U$ to the  $A$-submodule   
$\delta(U)$   of $\tau(U)$  defined by 
$\delta(U) =\{ (x, \ldots, x )\, : \, x\in U \}$  and which sends   
the morphism $f:  U\to V$ of $A'$-modules to the morphism  $\delta(f)$
defined to be the restriction of $\tau(f)$ to $\delta(U)$. Finally let 
$\epsilon: \modh A \to$ $\modh A' $  be the extension functor 
$\CO'\otimes_{\CO} -$. The functors $\epsilon$, $\delta$ and $\tau$ are
exact functors of $\CO$-linear categories where we regard $\modh A'$  
as an $\CO$-linear category by restriction of  scalars.

\begin{pro} \label{galois-natural}  
With the notation and assumptions above, the functors  
$\epsilon \circ \delta $ and $\tau $ are naturally isomorphic. 
\end{pro}

\begin{proof}   
Let $\alpha \in \CO'$ be such that $k'=k[ \bar \alpha ]$ where  
$ \bar \alpha =  \alpha +  J(\CO')  \in k' $. Then 
$\{ \bar \alpha^i\, :\, 0\leq i\leq d-1\}$ is a $k$-basis of $k'$.    
Since the extension $\CO \subseteq \CO' $ is unramified $J(\CO') =$
$ J(\CO) \CO' $.  Hence by  Nakayama's lemma  
$\{\alpha^i\, :\, 0\leq i\leq d-1\}$ is an $\CO$-basis of $\CO' $.
Let $U$ be a  finitely  generated $A'$-module  and let 
$$\eta_U : \epsilon\delta (U) \to \tau(U)$$
be the unique  $\CO'$-linear extension of the inclusion   
$\delta(U) \subseteq \tau(U)$. Then $\eta=(\eta_U)$ is a natural 
transformation  from $\epsilon\delta $ to $\tau $.   
We will show that $\eta $ is an isomorphism. It suffices to show that 
this is an $\CO$-linear isomorphism; that is, we may assume that 
$A=\CO$ and $A'=$ $\CO'$. We show first that $\eta_U$ is an 
isomorphism for $U=\CO'$. Since $\{\alpha^i\ :\ 0\leq i\leq d-1\}$ is 
an $\CO$-basis of $\CO'$, it follows that 
$\{(\alpha^i,\alpha^i,...,\alpha^i)\ :\ 0\leq i\leq d-1\}$ is an 
$\CO$-basis of $\delta(\CO')$. We claim that this set is an 
$\CO'$-basis of $\tau(\CO')$.  Since the cardinality of this set is 
equal to $d$, which is also the $\CO'$-rank of $\tau(\CO')$, it 
suffices to show that the image of  this set in 
$k'\ten_{\CO'} \tau(\CO')$ is linearly independent. For notational 
convenience, assume temporarily that $\CO'=k'$. Suppose that
$$\sum_{i=0}^{d-1}\ \lambda_i (\alpha^i,\alpha^i,..,\alpha^i)=0 \ $$
for some coefficients $\lambda_i\in$ $k'$. The scalar $\lambda_i$ acts 
on the $j$-th coordinate as multiplication by $\sigma^{-j}(\lambda_i)$, 
so this is equivalent to the $d$ equations
$$\sum_{i=0}^{d-1} \sigma^{-j}(\lambda_i)\alpha^i =0$$
for $0\leq j\leq d-1$. Applying $\sigma^j$ to the corresponding equation 
implies that this is equivalent to 
$$\sum_{i=0}^{d-1} \lambda_i \sigma^j(\alpha)^i = 0$$
for $0\leq j\leq d-1$.  Note that the $\sigma^j(\alpha)$, with 
$0\leq j\leq d-1$, are pairwise different, and hence the Vandermonde 
matrix $(\sigma^j(\alpha)^i)$ has nonzero determinant. 
Thus all coefficients $\lambda_i$ are $0$. 

Reverting to the ring $\CO'$ as before, this shows that $\eta_U$ is an 
isomorphism if $U=\CO'$. Since $\eta_U$ is additive in $U$, it follows 
that $\eta_U$ is an isomorphism whenever $U$ is free of finite rank 
over $\CO'$. Let $U$ be an arbitrary  finitely generated  $\CO'$-module 
and let $ Q_1\to Q_0 \to U \to 0$  be a free presentation of $U$. By the 
naturality of $\eta $ we obtain  the following commutative diagram

$$\xymatrix{ \epsilon\delta(Q_1) \ar[r]\ar[d] ^{\eta_{Q_1}} 
& \epsilon\delta(Q_0) \ar[r]\ar[d]^{\eta_{Q_0}}
&  \epsilon\delta(U) \ar[d]^{\eta_U} \ar[r] & 0\\
\tau(Q_1)\ar[r] & \tau(Q_0) \ar[r] & \tau(U) \ar[r] & 0 }$$

By the exactness of $\delta $, $\tau$ and $\epsilon$, the horizontal 
rows are exact. Since $Q_1 $ and $Q_0$ are $\CO'$-free, the vertical 
maps $\eta_{Q_1} $ and $\eta_{Q_0}$ are isomorphisms. It follows that  
$\eta_U $ is an isomorphism.
\end{proof}

Let  $G$ be a finite group. For $a=\sum_{g\in G} \alpha_g  g$ 
an element of $k'G$ or of $\CO' G$, denote by $k[a] $ the smallest 
subfield of $k'$ containing $k$ and  (the images in $k'$ of) all 
coefficients $\alpha_g$, $g\in G$. 

\begin{lem}  \label{minimal-galois} 
Let $G$ be finite group. Let $b'$ be a block of $k'G$ and $b$ a block  
of $kG$ such that $bb'  \ne 0 $. Suppose that $k'=k[b']$. Then the 
extension $k'/k$ is finite. Set $d=[k': k]$ and let $\sigma $ be a 
generator of $\Gal(k'/k)$. Then 
$$ b =\sum_{0\leq i\leq d-1}  \sigma^i( b')  $$  
is the   block  decomposition  of $b$ in $k'G$.
\end{lem}

\begin{proof}  
The block idempotent $b'$ of $k'G$ has coefficients contained in a 
finite subfield of $k'$ (because $G$ has a finite splitting field) and 
hence $k[b']/k$ is a finite extension.  For any $i $, 
$0\leq i \leq d-1 $, $\sigma^i(b')$ is a block of $k'G $ satisfying  
$\sigma^i(b') b= \sigma^i (b'b) \ne 0 $.  Hence we only need to show 
that $ \sigma^i(b') \ne b'$ for any $i $, $0\leq i \leq d-1 $. This 
follows from the fact that $k'=k[b']$ is a finite Galois extension 
with Galois group $\langle \sigma\rangle$. 
\end{proof} 

\begin{thm} \label{thm:further-descent}   
Suppose that $k$ and $k'$ are finite and that $\CO$ and $\CO'$  are 
absolutely  unramified. Let $G$, $H$ be finite groups, $b$ a block of 
$\OG$ and $c$ a block of $\OH$. Let $b'$ be a block of $\CO'G$ 
satisfying $bb'\neq$ $0$ and let $c'$ be a block of $\CO'H$ satisfying 
$cc'\neq$ $0$. Suppose that $k'=$ $k[b']=$ $k[c']$. The following hold.

\begin{enumerate} [(a)]  
\item 
If $\CO'Gb'$ and $\CO' Hc'$ are Morita equivalent via an 
$(\CO' Gb',\CO' Hc')$-bimodule  $M'$, then $\OGb$ and $\OHc$ are Morita 
equivalent via an $(\OGb,\OHc)$-bimodul $M$, such that $M'$ is 
isomorphic to a direct summand of  $\CO'\tenO M$. In particular, if $M'$ 
has a trivial source, then $M$ has a trivial source.

\item
If there is a Rickard complex $X'$ of $(\CO' Gb', \CO' H c')$-bimodules, 
then there is a Rickard complex $X$ of $(\OGb,\OHc)$-bimodules such that
$X'$ is isomorphic to a direct summand of $\CO'\tenO X$. In particular,
if $b'$ and $c'$ are splendidly Rickard equivalent, then $b$ and $c$ are
splendidly Rickard equivalent.

\item
If there is a virtual Morita equivalence (resp. a $p$-permutation 
equivalence) between $\CO'Gb'$ and $\CO'H c'$, then there is
a virtual Morita equivalence (resp. a $p$-permutation equivalence)
between $\OGb$ and $\OHc$.
\end{enumerate}
\end{thm}

\begin{proof}
Let $\sigma$ be a generator of $\Gal(k'/k)$.  Since $k'=$ $k[b']$ and 
$k'=$ $k[c']$, it follows from \ref{minimal-galois} that the 
$\sigma^i(b')$, $0\leq i\leq d-1$, are pairwise different blocks of
$\CO' G$ whose sum is $b$, and the analogous statement holds for 
$\CO' H c'$ and $c$.  Suppose that $\CO'Gb'$ and $\CO' Hc'$ are Morita 
equivalent via an $(\CO' Gb',\CO' Hc')$-bimodule  $M'$. Then
$\CO' G\sigma^i(b')$ and $\CO' H \sigma^i(c')$ are Morita equivalent
via the bimodule ${^{\sigma^i}{M'}}$. Thus the direct sum $\tau(M')=$
$\oplus_{i=0}^{d-1}\ {^{\sigma^i}{M'}}$  induces a Morita equivalence 
between $\CO' Gb$ and $\CO'H c$. By Proposition \ref{galois-natural}, 
the above direct sum is isomorphic to $\CO'\tenO M$ for some 
$(\OGb,\OHc)$-bimodule $M$. By Proposition \ref{extendProp}, $M$ induces 
a Morita equivalence. It follows from Lemma \ref{lemma-source} that if 
$M'$ has a trivial source, then $M$ has a trivial source. This proves 
(a). Obvious variations of this argument prove (b) and (c).  
\end{proof}

%%%%%%%%%%%%%%%%%%%%%%%%%%%%%%%%%%%%%%%%%
\section{On cyclic blocks} \label{cyclic-Section}

We prove in this section the Theorems \ref{cyclic-splendid-intro} and 
\ref{cyclic-splendid-general}. Modules in this section are finitely 
generated. Let $(K,\CO,k)\subseteq$ $(K',\CO',k')$ 
be an extension of $p$-modular  systems as in Definition 
\ref{defi:pmodext}.

Let $G$ be a finite group and $b$ a block of $\CO'G$ with a nontrivial
cyclic defect group $P$. If $k'$ is a splitting field for all subgroups
of $G$, then Rouquier constructed a $2$-sided splendid tilting
complex $X'$ of $(\CO'Gb, \CO' N_G(P)e)$-bimodules, where $e$ is the
Brauer correspondent of $b$. The hypotheses in \cite{Rouqcyclic}
also require $K'$ to be large enough, but  it is easy to see that 
Rouquier's construction works with $\CO'$  absolutely unramified. In 
order to prove Theorem \ref{cyclic-splendid-intro}, we need to show that 
Rouquier's complex is defined over the subring $\CO$ so 
long as the block idempotent $b$ is contained in $\OG$. We review
Rouquier's construction and other facts  on cyclic blocks as we go 
along. We start with some basic observations regarding automorphisms of 
Brauer trees.

\begin{rem} \label{tree-auto-Rem}
Let $G$ be a finite group and $b$ a block of $\OG$ with a nontrivial 
cyclic defect group $P$. Suppose in addition that $\CO$ contains a 
primitive $|G|$-th root of unity. Any 
ring automorphism $\sigma$ of $\OGb$ permutes the sets of isomorphism 
classes of simple modules, of projective indecomposable modules, and the 
set of ordinary irreducible characters of $\OGb$.
Thus $\sigma$ induces an automorphism of the Brauer tree of $b$.
If $|P|=2$, then $\OGb$ is Morita equivalent to $\CO C_2$, and the 
Brauer tree has a single edge and no exceptional vertex. Thus there 
are two automorphisms of this Brauer tree - the identity, and the 
automorphism exchanging the two vertices, and both are induced by ring 
automorphisms (the automorphism of $\CO C_2$ sending the nontrivial 
group element $t$ of $C_2$ to $-t$ in $\CO C_2$ exchanges the two 
vertices of the tree). If $|P|\geq$ $3$, then the Brauer tree has an 
exceptional vertex or at least two edges. In both cases, an easy 
combinatorial argument shows that an automorphism of the Brauer tree is 
uniquely determined by its effect on the edges of the tree. It follows 
that the automorphism of the Brauer tree induced by a ring automorphism
$\sigma$ of $\OGb$ is already determined by the induced ring
automorphism $\bar\sigma$ of $kG\bar b$, where $\bar b$ is the
image of $b$ in $kG$. This is the reason for why the following
Lemma, which is an immediate consequence of (the proof of)  
\cite[Proposition 4.5, Remark 4.6]{Listable}, is formulated over $k$ 
rather than $\CO$. 
\end{rem}

\begin{lem} \label{tree-auto}
Let $G$ be a finite group and $b$ a block of $kG$ with a nontrivial 
cyclic defect group $P$ of order at least $3$. Suppose that $k$ is a 
splitting field for the subgroups of $G$. Let $\gamma$ be a ring 
automorphism of $kGb$. Then $\gamma$ induces an automorphism of the 
Brauer tree of $b$ which fixes at least one vertex.
\end{lem} 

\begin{proof}
The statement is trivial if the Brauer tree has an exceptional
vertex (which is necessarily fixed). Suppose that the Brauer tree
does not have an exceptional vertex. Then $|P|=p\geq 3$, and the tree
has $p-1$ edges; note that $p-1$ is even. An easy argument shows that 
any tree automorphism fixes an edge or a vertex. In the latter case, 
we are done, so assume that it fixes an edge, which we label by $i$. 
Removing this edge from the Brauer tree yields two disjoint trees. If 
the two disjoint trees are exchanged by the Brauer tree automorphism, 
then they have the same number $t$ of edges. But then the number of 
edges of the Brauer tree itself is $2t+1$, which is odd, a 
contradiction. Thus the Brauer tree automorphism stabilises the two 
trees obtained from removing the edge $i$. But then it stabilises the
two vertices connected by $i$, whence the result.
\end{proof}

\begin{proof}[Proof of Theorem \ref{cyclic-splendid-intro}]
Since any block of a finite group algebra has a finite splitting
field, we may assume that $k$ and $k'$ are finite.

Rouquier's splendid Rickard complex is constructed inductively,
separating the cases according to whether $G$ has a nontrivial
normal $p$-subgroup or not. The construction of this splendid
equivalence is played back to \cite[Theorem 10.3]{Rouqcyclic}.
It suffices therefore to show that the complexes arising in that
theorem are defined over $\CO$. We start with the case $O_p(G)=$
$\{1\}$. Since $p$-permutation modules of finite groups lift uniquely, 
up to isomorphism, from $k$ to $\CO$, it is easy to see that we
may replace $\CO$ and $\CO'$ by $k$ and $k'$, respectively.
(This simplifies notation, but one could as well write the proof 
over $\CO$ and $\CO'$, if desired.)

Let $H$ be the normaliser in $G$ of the unique subgroup $Z$ of order
$p$ of $P$, and let $c$ be the block of $k'H$ corresponding to $b$
via the Brauer correspondence. Since any block idempotent of $kH$ is
contained in $kC_G(Z)$, we have $\Br_Z(b)=c$. Since $b\in$ $kG$, it 
follows that also $c\in$ $kH$. 

Set $A=$ $kGb$, $A'=$ $k'Gb$, $B=$ $kHc$ and $B'=$ $k'Hc$. By 
\cite[Theorem 10.3]{Rouqcyclic}, there is a splendid Rickard complex 
$X'$ of $(A', B')$-modules of the form
$$\xymatrix{\cdots\ar[r]&0\ar[r]& N'\ar[r]^{\pi'}
&M'\ar[r]&0\ar[r]&\cdots}$$
for some projective $(A', B')$-bimodule $N'$ and some bimodule
homomorphism $\pi'$ such that $(N',\pi')$ is a direct summand
of a projective cover of $M'$. The algebra $B'$ is Morita equivalent
to the serial algebra $k'(P\rtimes E)$, where $E$ is the inertial 
quotient of $b$. That is, the Brauer tree of $B'$ is a star with
$|E|$ edges, and exceptional vertex in the center, if any.
 
By Proposition \ref{extendProp}, in order to prove Theorem 
\ref{cyclic-splendid-intro}, it suffices to prove that there is a 
complex of $A$-$B$-bimodules $X$ satisfying $k'\tenk X\cong$ $X'$. 

The $(A, B)$-bimodule $bkGc$ has, up to isomorphism, a unique 
nonprojective indecomposable bimodule summand $M$. This bimodule
and its dual induce a stable equivalence of Morita type between
$A$ and $B$ (this goes back to Green \cite{Greenwalk}; see
\cite{Licyclic} for a proof using this terminology). As a
$k(G\times H)$-module, the diagonal subgroup $\Delta P$ is a vertex
of $M$. The analogous properties hold for $A'=$ $k'Gb$ and
$B'=$ $k'Hc$. Lemma \ref{lemma-c1} implies that if $M'$ is the
unique (up to isomorphism) nonprojective bimodule summand of
$bk'Gc$, then $M'\cong$ $k'\tenk M$. 

If $|P|=2$, then $M'$, hence $M$, induces a Morita equivalence, and
so we are done in that case. Suppose now that $|P|\geq$ $3$.

The bimodule $M'$ is the right term in Rouquier's complex. For the left 
term, we need to show that $N'\cong$ $k'\tenk N$ for some (by 
\ref{contract} necessarily projective) $(A, B)$-bimodule $N$, and that 
the map $\pi'$ is obtained from applying $k'\tenk-$ to some map 
$\pi : N\to$ $M$. To that end, we need to show that $N'$ 
is $\Gamma$-stable, where as before $\Gamma=$ $\Gal(k'/k)$. This
will follow from Rouquier's description of $N'$, which we review
briefly. 

For that purpose, we need some classical facts on blocks with cyclic 
defect groups which have their origins in work of Brauer, Dade, and 
Green. We follow the presentation given in \cite{LinModules}, 
\cite{Licyclic}.
Denote by $I$ a set of representatives of the conjugacy classes of
primitive idempotents in $A'$, and by $J$ a set of representatives of
the conjugacy classes of primitive idempotents in $B'$.
Set $S_i=$ $A'i/J(A')i$ for all $i\in$ $I$ and $T_j=$ $B'j/J(B')j$ for
all $j\in$ $J'$. 

Using general properties of stable equivalences of Morita type from 
\cite{Listable} and well-known facts on cyclic blocks, it follows that 
the $B'$-modules $\CF(S_i)=$ $M'^*\tenA S_i$ and the $A$-modules 
$\CG(T_j)=$ $M'\tenB T_j$ are indecomposable and uniserial. There are 
unique bijections $\delta, \gamma : I\to $ $J$ such that 
$T_{\delta(i)}$ is isomorphic to the unique simple quotient of 
$\CF(S_i)$ and such that $T_{\gamma(i)}$ is isomorphic to the unique 
simple submodule of $\CF(S_i)$.  For any $i\in$ $I$ there are unique 
uniserial submodules $U_i$ and $V_i$ of $A'i$ isomorphic to 
$\CG(T_{\delta(i)})$ and $\CG(\Omega(T_{\gamma(i)}))$, respectively.
There are unique permutations $\rho$ and $\sigma$ of $I$ such that
the top composition factors of $U_i$ and $V_i$ are isomorphic to
$S_{\rho(i)}$ and $S_{\sigma(i)}$, respectively. In particular,
$A'\rho(i)$ is a projective cover of $U_i\cong$ 
$M'\ten_{B'} T_{\delta(i)}$. Since $B'$ is symmetric, the projective 
indecomposable right $B'$-module $\delta(i)B'$ is a projective cover of 
the simple right $B'$-module $T_{\delta(i)}^\vee$. It follows from the 
description of projective covers of bimodules in 
\cite[Lemma 10.2.12]{Rouqcyclic}, that a projective cover of the 
$(A', B')$-bimodule $M'$ has the form
$$Z' = \oplus_{i\in I}\ A'\rho(i)\tenk \delta(i)B'$$
together with a surjective $(A', B')$-bimodule homomorphism $\pi'$ 
from  $Z'$ onto $M'$.

The permutations $\rho$ and $\sigma$ determine the Brauer tree as
follows. For $i\in$ $I$, denote by $i^\rho$ the 
$\langle\rho\rangle$-orbit of $i$ in $I$; use the analogous notation 
for $i^\sigma$.
The vertices of the Brauer tree are the $\langle\rho\rangle$-orbits and
$\langle\sigma\rangle$-orbits, with exactly one edge labelled $i$ 
linking $i^\rho$ and $i^\sigma$.
Denote by $v$ the exceptional vertex with exceptional multiplicity
$m$; if there is no exceptional vertex, we choose for $v$ a
$\Gamma$-stable vertex (which is possible by \ref{tree-auto}) 
and set $m=$ $1$. Note that there is a unique edge $\rho(i)$
which links $i^\rho=$ $\rho(i)^\rho$ and $\rho(i)^\sigma$. Since there 
is a unique minimal path from $v$ to any other vertex in the Brauer 
tree, it follows that we have a well-defined notion of distance from 
$v$ - this is the number of edges of a minimal path from $v$ to any 
other vertex.

The construction of Rouquier's bimodule complex is based on a
partition of $I$ into two subsets. Note that the vertex
$i^\rho=$ $\rho(i)^\rho$ is linked to the vertex $\rho(i)^\sigma$ via 
the edge labelled $\rho(i)$. Thus exactly one of these two vertices is
further away from $v$ than the other.
We denote by $I_0$ the set of all $i\in$ $I$ such
that the vertex $i^\rho$ of the Brauer tree is further away
from the exceptional vertex $v$ than the vertex $\rho(i)^\sigma$.
In particular $i^\rho$ is nonexceptional in that case.
We set $I_1=$ $I\setminus I_0$; that is, $I_1$ consists of all
$i\in$ $I$ such that $\rho(i)^\sigma$ is further away from $v$ than
$i^\rho$. In particular, $\rho(i)^\sigma$ is nonexceptional in that 
case. Then
$$N' = \oplus_{i\in I_1}\ A'\rho(i)\tenk \delta(i)B'\ $$
This is a direct summand of the above projective cover of $M'$, and
we denote the restriction of $\pi'$ again by $\pi'$. 
Since the action of $\Gamma$ on the Brauer tree fixes $v$, it
follows that the set $I_1$ is $\Gamma$-stable, and hence so is
the isomorphism class of $N'$. It follows from Lemma \ref{lemma-c4}
that there is a projective $(A,B)$-bimodule $N$ such that
$N'\cong$ $k'\tenk N$. To see that the map $\pi'$ can also be
chosen to be of the form $\Id_{k'}\ten \pi$ for some bimodule
homomorphism $\pi : N\to$ $M$, consider a projective cover
$\pi : Z\to$ $M$. Observe that then $k'\tenk Z\cong$ $Z'$ yields
the projective cover of $M'$ above, and Lemma \ref{lemma-c4}
implies that $Z$ has a summand isomorphic to $N$, so we just need
to restrict $\pi$ to $N$ and the extend scalars to $k'$. 

This shows that $A$ and $B$ are splendidly Rickard equivalent. It 
remains to show that the complex in \cite[Theorem 10.3]{Rouqcyclic}
is also defined over $k$ in the case where $O_p(G)$ is nontrivial. 

Set $R=$ $O_p(G)$ and assume that $R\neq$ $\{1\}$. If $R=$ $P$, there 
is nothing further to prove; thus we may assume that $R$ is a proper 
subgroup of $P$. Let $Q$ be the unique subgroup of $P$ such that 
$|Q:R|=$ $p$. Changing earlier notation, set $H=$ $N_G(Q)$, and denote 
by $c$ the block of $k'H$ which is the Brauer correspondent of $b$. We 
have $c=$ $\Br_Q(b)$, and hence $c\in$ $kH$. Set $A=$ $kGb$, $A'=$ 
$k'Gb$, $B=$ $kHc$, and $B'=$ $k'Hc$.

Note that $kG\tenkR kH\cong$ $\Ind^{G\times H}_{R}(k)$ as
$k(G\times H)$-modules. Thus $A\tenkR B$, together with the 
multiplication map $A\tenkR B\to$ $bkGc$, is a relatively 
$\Delta R$-projective presentation of $bkGc$, where we regard this 
bimodule as $k(G\times H)$-module. Thus some bimodule summand of 
$A\tenkR B$ yields a relatively $\Delta R$-projective cover of $bkGc$.

Rouquier's splendid Rickard complex of $(A', B')$-bimodules from
\cite[Theorem 10.3]{Rouqcyclic} is in the present situation a complex 
$X'$ of the form 
$$\xymatrix{\cdots\ar[r]&0\ar[r]& N'\ar[r]^{\pi'}
&M'\ar[r]&0\ar[r]&\cdots}$$
which is a direct summand of the complex
$$\xymatrix{\cdots\ar[r]&0\ar[r]& A'\ten_{k'R} B' \ar[r]^{\pi'}
&bk'Gc\ar[r]&0\ar[r]&\cdots}$$
where $\pi'$ is the map induced by multiplication, where $M'$ is the 
unique (up to isomorphism) indecomposable direct bimodule summand of 
$bk'Gc$ with vertex $\Delta P$, and where either $N'=\{0\}$ or 
$(N',\pi')$ is a relatively $\Delta R$-projective cover of $M'$.  
As before, Lemma 
\ref{lemma-c1} implies that $M'\cong$ $k'\tenk M$, where $M$ is the 
unique indecomposable direct bimodule summand of $bkGc$ with vertex 
$\Delta P$. If $N'=$ $\{0\}$, then $X'$ is the complex $M'$ 
concentrated in degree $0$, so is trivially of the form $k'\tenk X$, 
where $X$ is the complex $M$ concentrated in degree $0$. If 
$N'\neq$ $\{0\}$, then $N'$ is a relatively $\Delta R$-projective 
cover of $M'$. The properties collected in Lemma \ref{lemma-c3} imply 
that this relative projective cover is isomorphic to one obtained
from extending the scalars in a relatively $\Delta R$-projective
cover of $M$, and hence in this case we also get that $X'\cong$
$k'\tenk X$ for some complex $X$. This completes the proof of
Theorem \ref{cyclic-splendid-intro}.
\end{proof} 

\begin{proof}[{Proof  of Theorem~\ref{cyclic-splendid-general}}] 
Denote by $\bar b$ the image of $b$ in $kG$. Write $\bar  b =$
$\sum_{g\in G }\alpha_g g $ with coefficients $\alpha_g\in$ $k$. By  
Proposition~\ref{extendProp} we may assume that $k =\Fp[\bar b]$. Since 
all central idempotents of $\CO G$ belong to $W(k) G$, we may assume 
that $\CO$ is absolutely  unramified. Let $\tilde k$ be a splitting 
field for $G$ containing $k$, and let $\bar b' $ be a block of 
$\tilde kG$ such that $\bar b\bar b'\ne 0$. Set $\tilde\CO=W(\tilde k)$ 
and let $\tilde K$ be the field of fractions of $\tilde \CO $. Let 
$k' =k [\bar b']  \subseteq \tilde k $,  let $\CO' =W(k')$ and let 
$K' $ the field of fractions of $\CO'$. Let $b'$ be the  block of 
$\CO'G $ lifting $\bar b'$. By Lemma~\ref{minimal-galois} $P$ is a 
defect group of $b' $. Let $e'$ be the block of $\CO' N_G(P)$ in 
Brauer correspondence with $b'$. Then $\bar e' \bar e \ne 0 $, 
$k={\mathbb F}_p[\bar e]$ and $k'= k[\bar e']$. By 
Theorem~\ref{cyclic-splendid-intro},  applied to the block $b'$ and 
the extension of $p$-modular systems $(K',\CO', k') \subseteq$   
$(\tilde K, \tilde\CO, \tilde k )$ there is a splendid Rickard complex 
$X'$ of $(\CO' G b', \CO' N_G(P) e')$-bimodules. It follows from 
Theorem \ref{thm:further-descent} that there is a splendid
Rickard complex $X$ of $(\OGb,\CO N_G(P)e)$-bimodules, whence the
result.
\end{proof}

\begin{rem}
Zimmermann showed in \cite{Zimm97} that Rouquier's complex can be 
extended to Green orders, a concept due to Roggenkamp \cite{Rogg92}. 
This might  provide  alternative proofs of the  Theorems 
\ref{cyclic-splendid-intro} and \ref{cyclic-splendid-general}. 
In order to apply Zimmermann's result one would need to show that 
$\OGb$ and $\CO N_G(P)e$ are Green orders whose underlying 
structure data, as required in \cite{Zimm97}, coincide. 
\end{rem}

%%%%%%%%%%%%%%%%%%%%%%%%%%%%%%%%%%%%%%%%%%%%%%%%%%
\section{Descent for Morita equivalences with endopermutation source} 
\label{basic-Section}
  
We briefly recall some notation and general facts about   
endopermutation modules over $p$-groups, which we will use without
further reference. Let $(K, \CO, k)$ be a $p$-modular system, and let
$P$ be a finite $p$-group. By an endopermutation $\OP$-module we will
always mean an endopermutation $\CO$-lattice.
  
By results of Dade \cite{Dadeendo}, the tensor  product of two 
indecomposable endopermutation $\OP$-modules (respectively 
$kP$-modules)  with vertex $P$ has a  unique  indecomposable direct 
summand with vertex $P$; this induces an abelian group structure on the 
set of isomorphism classes of indecomposable endopermutation 
$\OP$-modules (respectively $kP$-modules) with vertex $P$.
The resulting group is denoted $D_\CO(P)$ (respectively $D_{k}(P)$), 
called the {\it Dade group of $P$ over $\CO$ (respectively  $k$)}.      
Let $V$ be an endopermutation $\OP$-module (respectively $kP$-module) 
having an indecomposable direct summand with vertex $P$. For any 
subgroup $Q$ of $P$, the indecomposable direct summands of 
$\Res^P_Q(V)$ with vertex $P$ are all isomorphic, and we  denote by 
$V_Q$ an indecomposable direct summand of $\Res^P_Q(V)$ with vertex 
$Q$. If $V$ is an indecomposable endopermutation $\OP$-module, then
$\bar V=$ $k\tenO V$ is an indecomposable endopermutation 
$kP$-module with the same vertices. 
 
Let $\CF$ be a saturated fusion system on $P$. Following the 
terminology in \cite[3.3]{LiMa} we say that the class $[V]$ of an 
endopermutation $\OP$-module (respectively $kP$-module) $V$ in the 
Dade group $D_\CO(P)$ is {\it $\CF$-stable} if for every isomorphism 
$\varphi : Q\to$ $R$ in $\CF$ between two subgroups $Q$, $R$ of $P$ we 
have $V_Q\cong$ ${_\varphi{V_R}}$. Here ${_\varphi{V_R}}$ is the 
$\OQ$-module (respectively $kQ$-module) which is equal to $V_R$ as an 
$\CO$-module (respectively $k$-module), with $u\in$ $Q$ acting as 
$\varphi(u)$ on $V_R$. The $\CF$-stable classes of indecomposable 
endopermutation $\OP$-modules (respectively $kP$-modules) with vertex 
$P$ form a subgroup of $D_{\CO}(P)$ (respectively $D_k (P))$, denoted 
$D_{\CO}(P,\CF)$ (respectively $D_{k}(P, \CF)$).   
The $\CF$-stability of the class $[V]$ is a slightly weaker condition
than the $\CF$-stability of the actual module $V$. More precisely, an
$\OP$-module $V$ is {$\CF$-stable} if for every isomorphism 
$\varphi : Q\to$ $R$ in $\CF$ between two subgroups $Q$, $R$ of $P$ we 
have $\Res^P_Q(V)\cong$ ${_\varphi{\Res^P_R(V)}}$. If $V$ is an
$\CF$-stable endopermutation $\OP$-module having an indecomposable
direct summand $V_P$ with vertex $P$, then the class $[V_P]$ in
$D_\CO(P)$ is clearly $\CF$-stable. We will need the following result.

\begin{pro}[{\cite[Proposition 3.7]{LiMa}}]  \label{LiMaProp37}
Let $P$ be a finite $p$-group, $\CF$ a saturated fusion system on
$P$ and $V$ and indecomposable endopermutation $\OP$-module with
vertex $P$ such that the class of $[V]$ in $D_\CO(P)$ is $\CF$-stable.
Then there exists an $\CF$-stable endopermutation $\OP$-module $V'$ 
having a direct summand isomorphic to $V$. Moreover, we may choose
$V'$ to have $\CO$-rank prime to $p$, and the analogous result
holds with $k$ instead of $\CO$.
\end{pro}

The statement on the rank is not made explicitly in 
\cite[Proposition 3.7]{LiMa}, and this Proposition is stated there 
only over $k$, but the slightly stronger version above follows 
immediately from the construction of $V'$ in the proof of that 
Proposition in \cite{LiMa}. 

%For a saturated fusion system $\CF$ on $P$ we denote by $\foc(\CF)$ the 
%subgroup of $P$ generated by elements of the form $\varphi(u)u^{-1}$, 
%where $u\in$ $P$ and $\varphi\in$ $\Hom_\CF(\langle u\rangle, P)$ (see 
%\cite[Part I]{AKO} for more details). The $\CF$-stable linear characters 
%of $P$ form a subgroup of $D_\CO(P,\CF)$ which we identify with the 
%group $\Hom(P/\foc(\CF),\CO^\times)$.  
 
For $P$, $Q$ finite $p$-groups, $\CF$ a fusion system on $P$ and 
$\varphi : P\to$ $Q$ a group isomorphism, we set $\Delta\varphi=$ 
$\{(u,\varphi(u))\ |\ u\in P\}$, and we denote by ${^\varphi{\CF}}$ the 
fusion system on $Q$ induced by $\CF$ via the isomorphism $\varphi$. 
That is, for $R$, $S$ subgroups of $P$, we have 
$\Hom_{{^\varphi{\CF}}}(\varphi(R),\varphi(S))=$
$\varphi\circ\Hom_\CF(R,S)\circ\varphi^{-1}$, where we use the same 
notation $\varphi$, $\varphi^{-1}$ for their restrictions to $S$, 
$\varphi(R)$, respectively. 

The proof of Theorem \ref{thm:basicunramified} requires the following 
Lemma, due to Puig, which summarises some of the essential properties of
stable equivalences of Morita type with endopermutation source.
We assume in the remainder of this section that $k$ is large enough
for all finite groups and their subgroups, so that fusion systems
of blocks are saturated.

\begin{lem}[{\cite[7.6]{Pu99}}] \label{Vvarphi-existence}
Let $G$, $H$ be finite groups, $b$, $c$ blocks of $\OG$, $\OH$ with 
defect groups $P$, $Q$, respectively, and let $i\in$ $(\OGb)^P$ and 
$j\in$ $(\OHc)^P$ be source idempotents. Denote by $\CF$ the fusion 
system on $P$ of $b$ determined by $i$, and denote by $\CG$ the fusion
system on $Q$ determined by $j$. Let $M$ be an indecomposable 
$(\OGb, \OHc)$-bimodule inducing a stable equivalence of Morita type 
with endopermutation source.

Then there is an isomorphism $\varphi : P\to$ $Q$ and an indecomposable 
endopermutation $\Delta\varphi$-module $V$ such that $M$ is isomorphic 
to a direct summand of 
$$\OG i\tenOP \Ind_{\Delta\varphi}^{P\times Q}(V) \tenOQ j\OH$$
as an $(\OGb,\OHc)$-bimodule. For any such $\varphi$ and
$V$, the following hold.

\begin{enumerate} [\rm (a)]

\item  
$\Delta\varphi$ is a vertex of $M$ and $V$ is a source of $M$.

\item  
We have ${^\varphi{\CF}}=$ $\CG$, and when regarded as an $\OP$-module 
via the canonical isomorphism $P\cong$ $\Delta\varphi$, the class 
$[V]$ of the endopermutation module $V$ is $\CF$-stable.
\end{enumerate}
\end{lem}

See also  \cite[9.11.2]{LiBook} for a proof of the above Lemma. 

\begin{lem} \label{lem:lift}
Let $G$, $H$ be finite groups, $b$, $c$ blocks of $\OG$, $\OH$, 
respectively, with a common defect group $P$, and let $i\in$ 
$(\OGb)^{\Delta P}$ and $j\in$ $(\OHc)^{\Delta P}$ be source 
idempotents. Suppose that $i$ and $j$ determine the same fusion system 
$\CF$ on $P$. Let $V$ be an indecomposable endopermutation
$\OP$-module with vertex $P$ such that $[V]$ is $\CF$-stable. Consider 
$V$ as an $\CO\Delta P$-module via the canonical isomorphism 
$\Delta P\cong$ $P$. Set 
$$X = \OG i\tenOP \Ind^{P\times P}_{\Delta P}(V) \tenOP j \OH$$
The canonical algebra homomorphism
$$\End_{\CO(G\times H)}(X) \to \End_{k(G\times H)}(k\tenO X)$$
is surjective. In particular, for any direct summand $\bar M$ of 
$k\tenO X$ there is a direct summand $M$ of $X$ such that
$k\tenO M\cong$ $\bar M$.
 \end{lem}
 
\begin{proof}
Set $A=$ $i\OG i$ and $B=$ $j\OH j$, and set
$$U =  A \tenOP \Ind^{P\times P}_{\Delta P}(V) \tenOP B\ .$$ 
Set $\bar A=$ $k\tenO A$, $\bar B=$ $k\tenO B$, and $\bar U=$ 
$k\tenO U$. Since multiplication by a source idempotent is a Morita 
equivalence, it suffices to show that the canonical map
$$\End_{A\tenO B^\op}(U)\to \End_{\bar A\tenk \bar B^\op}(\bar U)$$
is surjective. By a standard adjunction, we have an isomorphism
$$\End_{A\tenO B^\op}(U) \cong \Hom_{\CO\Delta P}(V, U)\ .$$
Thus we need to show that the right side maps onto the corresponding
expression over $k$. The right side is equal to the space of
$\CO\Delta P$-fixed points in the module
$$\Hom_{\CO}(V, U) \cong V^\vee\tenO U\ .$$
Thus we need to show the surjectivity of the canonical map 
$$(V^\vee\tenO U)^{\Delta P}\to (\bar V^\vee\tenk \bar U)^{\Delta P}
\ .$$
Fixed points in a permutation module with respect to a finite group
action over either $\CO$ or $k$ are spanned by the orbit sums of a 
permutation basis, and hence for the surjectivity of the previous map 
it suffices to show that 
$$V^\vee\tenO U$$
is a permutation $\CO\Delta P$-module. By 
\cite[Proposition 4.1]{Linendo}, as an $\CO\Delta P$-module, $U$ is an 
endopermutation $\CO\Delta P$-module having $V$ as a direct summand. 
Thus $U^\vee\tenO U$ is a permutation $\CO\Delta P$-module having 
$V^\vee \ten U$ is a direct summand. In particular, $V^\vee\tenO U$ is
a permutation $\CO\Delta P$-module as required. The last statement
follows from lifting idempotents.
\end{proof}
  
As a consequence of the classification theorem of endopermutation 
modules over finite $p$-groups, if $U$ is an endopermutation 
$kP$-module having an indecomposable direct summand with vertex $P$, 
then there exists an endopermutation $\OP$-module $V$ such that
$\bar V\cong  U$  (see  \cite[Theorem 14.2]{ThevLausanne}). In 
particular, the canonical map $D_\CO(P)\to$ $D_k(P)$ is surjective
(cf. \cite[Corollary 8.5]{Bouc1}). By standard properties of 
endopermutation modules, the kernel of this map is $\Hom(P,\CO^\times)$.
Further, for any  saturated fusion system $\CF$ on $P$, if $V$ or
its class $[V]$ is $\CF$-stable, then  $\bar V$  or its class
$[\bar V]$ is $\CF$-stable, respectively. In particular, the surjection 
$D_\CO(P)\to$ $D_k(P)$ restricts to a map from $D_\CO(P,\CF)$ to 
$D_k(P,\CF)$. For fusion systems of finite groups, the following
result has also been observed by Lassueur and Th\'evenaz 
in \cite[Lemma 4.1]{LaTh}.

\begin{lem}\label{lem:stab}   
Let $P$ be a finite $p$-group and $\CF$ a saturated fusion system on 
$P$. The canonical map $D_\CO(P,\CF)\to$ $D_k(P,\CF)$ is surjective.
\end{lem} 

\begin{proof}  
Let $U$ be an indecomposable endopermutation $kP$-module with vertex 
$P$ such that the class $[U]$ of $U$ is in $D_k(P;\CF)$. By Proposition 
\ref{LiMaProp37} there is an $\CF$-stable endopermutation $kP$-module 
$U'$ of dimension prime to $p$ having a direct summand isomorphic to 
$U$. By the remarks at the beginning of this section, there is an 
endopermutation $\OP$-module $V'$ of determinant $1$ such that 
$\bar V'\cong$ $U$. Moreover, the determinant $1$ condition implies 
that $V'$ is unique up to isomorphism (see e. g. 
\cite[Lemma (28.1)]{Thev}). Then, for $Q$ a subgroup of $P$, the
$\OQ$-module $\Res^P_Q(V')$ of $V'$ is also the unique - up to 
isomorphism - lift of the $kQ$-module $\Res^P_Q(U')$ with determinant 
$1$, and for $\varphi : Q\to$ $R$ an isomorphism in $\CF$, the 
$\OQ$-module ${_\varphi{\Res^P_R(V')}}$ is the lift with determinant $1$
of the $kQ$-module ${_\varphi{\Res^P_R(U')}}$. Thus the $\CF$-stability 
of $U'$ implies that $V'$ is an $\CF$-stable $\OP$-module. But then the 
class of $V$ is $\CF$-stable in $D_\CO(P;\CF)$. By construction, we 
have $\bar V\cong$ $U$, proving the result.
\end{proof}

\begin{proof}[{Proof  of Theorem~\ref{thm:basicunramified}}]
Let $\bar M$ be an indecomposable $(kG\bar b,kH\bar c)$-bimodule 
inducing a Morita equivalence (resp. stable equivalence of Morita type).
Suppose that $\bar M$ has  endopermutation source $\bar V$. By Lemma 
\ref{Vvarphi-existence}, we may identify a defect group $P$ of $b$ with 
a defect group of $c$, such that $\bar M$ is a direct summand of
$$kG\bar i\tenkP \Ind^{P\times P}_{\Delta P}(\bar V)\tenkP \bar j kH$$
for some source idempotents $\bar i$, $\bar j$ of $\bar b$, $\bar c$.
Moreover, still by Lemma \ref{Vvarphi-existence}, these two source 
idempotents determine the same fusion system $\CF$ on $P$, and the class 
$[\bar V]$ in $D_k(P)$ is $\CF$-stable, where here $\bar V$ is regarded 
as a $kP$-module. By Lemma \ref{lem:stab} there is an endopermutation 
$\OP$-module $V$ satisfying $\bar V\cong$ $k\tenO V$ such that $[V]$ 
is $\CF$-stable in $D_\CO(P)$. It follows from Lemma \ref{lem:lift} 
that there is a direct summand $M$ of
$$\OG i\tenOP \Ind^{P\times P}_{\Delta P}(V)\tenOP  j \OH$$
satisfying $\bar M\cong$ $k\tenO M$, where $i$, $j$ are source 
idempotents lifting $\bar i$, $\bar j$. By construction, $M$ has vertex 
$\Delta P$ and source $V$, and by Proposition \ref{extendProp}, $M$ 
induces a Morita equivalence (resp. stable equivalence of Morita type). 
This proves (a). A Morita equivalence (resp. stable equivalence of 
Morita type) between $\OGb$ and $\OHc$ with endopermutation source 
induces clearly such an equivalence with endopermutation source between 
$kG\bar b$ and $kH\bar c$, and by (a), this lifts back to an equivalence
between $W(k)Gb$ and $W(k)Hc$ with the properties as stated in (b).
\end{proof} 

%%%%%%%%%%%%%%%%%%%%%%%%%%%%%%%%%%%%%%%%%%

\end{document}